\documentclass{amsart}
\usepackage{amsfonts,amsmath,amssymb}
\usepackage{mathrsfs,mathtools}
\usepackage{enumerate}
\usepackage{xspace,MACROS/mydef}
\usepackage{hyperref}
\usepackage{esint}
\usepackage{graphicx}
\DeclareMathAlphabet{\mathpzc}{OT1}{pzc}{m}{it}

\usepackage{color}
\newcommand{\EO}[1]{\textcolor{black}{#1}}

\newtheorem{theorem}{Theorem}[section]
\newtheorem{lemma}[theorem]{Lemma}
\newtheorem{proposition}{Proposition}[section]

\theoremstyle{definition}
\newtheorem{definition}[theorem]{Definition}
\theoremstyle{remark}
\newtheorem{remark}[theorem]{Remark}
\numberwithin{equation}{section}

\begin{document}

\title[A first--degree FEM for a control problem of fractional operators]{A \EO{piecewise linear} FEM for an optimal control problem of fractional operators: error analysis on \EO{curved domains}}

\author[E.~Ot\'arola]{Enrique Ot\'arola}
\address{Departamento de Matem\'atica, Universidad T\'ecnica Federico Santa Mar\'ia, Valpara\'iso, Chile.}
\email{enrique.otarola@usm.cl}
\thanks{EO has been supported in part by CONICYT through FONDECYT project 3160201 and Anillo ACT1106.}



\begin{abstract}
We propose and analyze a new discretization technique for a linear--quadratic optimal control problem involving the fractional powers of a symme\-tric and uniformly elliptic second oder operator; control constraints are considered. Since these fractional operators can be realized as the Dirichlet-to-Neumann map for a nonuniformly elliptic equation, we recast our problem as a nonuniformly elliptic optimal control problem. The rapid decay of the solution to this problem suggests a truncation that is suitable for numerical approximation. We propose a fully discrete scheme that is based on piecewise linear functions on quasi-uniform meshes to approximate the optimal control and first--degree tensor product functions on anisotropic meshes for the optimal state variable. \EO{We provide an a priori error analysis that relies on derived H\"older and Sobolev regularity estimates for the optimal variables and error estimates for an scheme that approximates fractional diffusion on curved domains; the latter being an extension of previous available results.} The analysis is valid in any dimension. We conclude by presenting some numerical experiments that validate the derived error estimates.
\end{abstract}

\maketitle

\section{Introduction}
\label{sec:introduccion}
In this work we shall be interested in the analysis of a new solution technique for a linear--quadratic optimal control problem involving the fractional powers of an uniformly elliptic second oder operator; control constraints are considered. Before describing such a discretization technique, we follow \cite{AO} and illustrate the PDE--constrained optimization problem we are interested in. Let $\Omega$  be an open and bounded domain in $\R^n$ ($n\ge1$) with Lipschitz boundary $\partial\Omega$. Given $s \in (0,1)$ and a desired state $\usf_d: \Omega \rightarrow \mathbb{R}$, we define the cost functional
\begin{equation}
\label{def:J}
 J(\usf,\zsf)= \frac{1}{2} \| \usf - \usf_{d} \|^2_{L^2(\Omega)} + \frac{\vartheta}{2} \| \zsf\|^2_{L^2(\Omega)}, 
\end{equation}
where $\vartheta > 0$ denotes the so--called regularization parameter. We shall be concerned with the following PDE--constrained optimization problem: Find
\begin{equation}
\label{eq:J}
 \text{min }J(\usf,\zsf),
\end{equation}
subject to the \emph{fractional state equation}
\begin{equation}
\label{eq:fractional}
\mathcal{L}^s \usf = \zsf  \text{ in } \Omega, \qquad \usf = 0   \text{ on } \partial \Omega, \\
\end{equation}
and the \emph{control constraints}
\begin{equation}
 \label{eq:cc}
\asf \leq \zsf(x') \leq \bsf \quad\textrm{a.e.}\quad x' \in \Omega . 
\end{equation}
The bounds $\asf , \bsf \in \mathbb{R}$ satisfy the property $\asf < \bsf$. The operator $\calLs$, with $s \in (0,1)$, is a fractional power of the second order, symmetric and uniformly elliptic operator 
\begin{equation}
\label{eq:L}
 \mathcal{L} w= - \DIV_{x'} (A(x') \nabla_{x'} w) 
\end{equation}
supplemented with homogeneous Dirichlet boundary conditions. The matrix of coefficients $A(x')$ is bounded and measurable in $\Omega$ and satisfies a uniform ellipticity condition. For convenience, we will refer to the optimal control problem \eqref{eq:J}--\eqref{eq:cc} as the \emph{fractional optimal control problem}.

Recently, the analysis of regularity properties of the solution to fractional PDE has received a tremendous attention: fractional diffusion has been one of the most studied topics in the past decade \cite{CS:07,MR2680400,Pablo,NOS3,MR2270163}. Concerning applications, capturing the essential behavior of fractional diffusion is fundamental in science and engineering since it allows for understanding of applications where anomalous diffusion is considered \cite{Abe2005403}, complex phenomena in mechanics \cite{atanackovic2014fractional}, turbulence \cite{wow}, nonlocal electrostatics \cite{ICH}, image processing \cite{Gatto} and finance \cite{MR2064019}. In many of these applications, control problems arise naturally \cite{AO,AOS}.

Regarding the numerical analysis literature, the approaches advocated for fractional diffusion are mainly divided in two categories. The first one is based on an approximated spectral representation that requires the solution of a large eigenvalue problem \cite{ILTA:05,MR2300467,YTLI:11}. The second approach is based on the singular integral definition of fractional diffusion: the obtained matrices are dense and the presence of a non-integrable kernel requires special attention \cite{AB:15,HO:14,Tobias}. As opposed to these approaches and inspired in the breakthrough by L. Caffarelli and L. Silvestre \cite{CS:07}, the authors of \cite{NOS} have designed and analyzed a new solution technique for fractional diffusion via an extension problem. Although the analysis of the proposed method is intricate, its implementation is done using standard components of finite element analysis \cite{NOS}.

In contrast, the analysis of optimal control problems involving fractional diffusion and nonlocal operators is still in its infancy; see \cite{AO, AOS,Marta,MR3472639} for some preliminary results. In \cite{Marta}, a PDE--constrained optimization problem involving a nonlocal diffusion equation is considered and analyzed. For such a problem, the authors propose a numerical scheme to approximate its solution and derive convergence results. In \cite{AO}, the authors have used the Caffarelli--Silvestre extension \cite{CS:07}, and its developments to both bounded domains \cite{BCdPS:12,CDDS:11} and a general class of elliptic operators \cite{Pablo,ST:10}, to design and analyze solution techniques for the \emph{fractional optimal control problem} \eqref{eq:J}--\eqref{eq:cc}. The analysis hinges on the fact that the fractional powers of $\mathcal{L}$  can be realized as an operator that maps a Dirichlet boundary condition to a Neumann condition via an extension problem on $\C = \Omega \times (0,\infty)$. This extension is the following local boundary value problem (see \cite{BCdPS:12,CS:07,CDDS:11,ST:10} for details):
\begin{equation}
\label{eq:alpha_harm_aux}
\mathcal{L}\ue - \tfrac{\alpha}{y}\partial_{y}\ue - \partial_{yy}\ue = 0 \text{ in } \C, 
\quad
\ue = 0 \text{ on } \partial_L \C, 
\quad
\partial_{\nu}^{\alpha} \ue= d_s \zsf \text{ on } \Omega \times \{0\},
\end{equation}
where $\partial_L \C= \partial \Omega \times [0,\infty)$ is the lateral boundary of $\C$, $\alpha = 1-2s \in (-1,1)$, $d_s = 2^{\alpha}\Gamma(1-s)/\Gamma(s)$ and 
the conormal exterior derivative of $\ue$ at $\Omega \times \{ 0 \}$ is
\begin{equation}
\label{def:lf}
\partial_{\nu}^{\alpha} \ue = -\lim_{y \rightarrow 0^+} y^\alpha \ue_y.
\end{equation}
The limit in \eqref{def:lf} is understood in the distributional sense \cite{CS:07,CDDS:11,ST:10}. We will call $y$ the \emph{extended variable} and call the dimension $n+1$ in $\R_+^{n+1}$ the \emph{extended dimension} of problem \eqref{eq:alpha_harm_aux}. As noted in \cite{BCdPS:12,CS:07,CDDS:11,ST:10}, the fractional powers of the operator $\mathcal{L}$ and the Dirichlet-to-Neumann operator of problem \eqref{eq:alpha_harm_aux} are related by 
\begin{equation*}
  d_s \calLs u = \partial_{\nu}^{\alpha} \ue \quad \text{in } \Omega.
\end{equation*}
Note that the differential operator in \eqref{eq:alpha_harm_aux} is
$
  -\DIV \left( y^{\alpha} \mathbf{A} \nabla \ue \right),
$
where $\mathbf{A}(x',y)$ $= \textrm{diag} \{A(x'),1\}$ for all $(x',y)$ in $\C $. Thus, we rewrite \eqref{eq:alpha_harm_aux} as 
\begin{equation}
\label{eq:alpha_harm}
  -\DIV \left( y^{\alpha} \mathbf{A} \nabla \ue \right) = 0  \textrm{ in } \C, \quad
  \ue = 0 \text{ on } \partial_L \C, \quad
  \partial_{\nu}^{\alpha} \ue  = d_s \zsf  \text{ on } \Omega \times \{0\}. 
\end{equation}

Invoking such a localization result, the authors of \cite{AO} overcome the nonlocality of $\calLs$ by considering an equivalent formulation for \eqref{eq:J}--\eqref{eq:cc}: $\text{min }J(\ue|_{y=0},\zsf)$ subject to the \emph{linear} equation \eqref{eq:alpha_harm} and \eqref{eq:cc}; the \emph{extended optimal control problem}. On the basis of this formulation, the following simple strategy to solve \eqref{eq:J}--\eqref{eq:cc} is proposed: given $s\in (0,1)$ and a desired state $\usf_d: \Omega \rightarrow \R$, solve the extended optimal control problem, thus obtaining $\bar{\zsf}(x')$ and $\bar{\ue}(x',y)$. Setting $\bar{\usf}: x' \in \Omega  \mapsto \bar{\usf}(x') = \bar{\ue}(x',0) \in \R$, the optimal pair $(\bar{\usf},\bar{\zsf})$ solving \eqref{eq:J}-\eqref{eq:cc} is obtained. Two numerical techniques are analyzed in \cite{AO}: one that is semidiscrete where the optimal control is not discretized, and the other one that is fully discrete and discretizes the optimal control using piecewise constant functions. A priori error estimates are derived; the ones for the fully--discrete scheme being quasi--optimal in terms of approximation. 

In this work, we continue with our recent research program by proposing and analyzing a new solution technique for the optimal control problem \eqref{eq:J}--\eqref{eq:cc}. This technique corresponds to a fully discrete scheme, which, in contrast to \cite{AO}, discretizes the optimal control with \emph{piecewise linear} functions on quasi-uniform meshes. Following \cite{NOS}, the optimal state is discretized with first--degree tensor product finite elements on anisotropic meshes. \EO{At this point, it is important to comment that although the essential localization results have been already introduced and analyzed in \cite{AO}, the error analysis of our proposed scheme comes with its own set of difficulties. Overcoming them has required us to provide several new, nontrivial, results. Let us briefly detail some of them:
\begin{enumerate}[1.]
 \item We provide regularity results, in H\"older and Sobolev spaces, for the continuous optimal control variables that require some suitable assumptions on the smoothness of the domain $\Omega$ and the matrix of coefficients $A$.
 \item In order to invoke the regularity results detailed in the previous point, we have seen the need of developing, inspired in \cite{NOS}, an a priori error analysis for fractional diffusion on curved domains $\Omega$ of class $C^2$. The analysis is able to deal with both the natural anisotropy of the mesh in the extended variable $y$ and the nonuniform coefficient $y^{\alpha}$ with $\alpha \in (-1,1)$.
 \item We provide an a priori error analysis for our scheme that requires a subtle interplay between Sobolev and H\"older regularity of the optimal control variables, as well as as growth conditions, and relies on the assumption that the boundary of the active sets consists of a finite number of rectifiable curves \cite{MV:08,MR:04}.
\end{enumerate}}

\EO{It is thus instructive to compare the error estimate derived in this work with the one obtained in \cite[Corollary 5.17]{AO}. The aforementioned error estimates read
\[
 \| \ozsf - \bar{Z} \|_{L^2(\Omega)} \lesssim |\log N|^{2s} N^{-\frac{1}{n+1}},\quad \textrm{and} \quad
 \| \ozsf - \bar{Z} \|_{L^2(\Omega)} \lesssim |\log N|^{2s} N^{-\frac{1}{n+1}\left( \frac{1}{2} + \theta \right)},
\]
respectively. Here $\ozsf$ denotes the optimal control, $\bar{Z}$ denotes the corresponding control approximation and $N$ denotes the number of degrees of freedom of the underlying mesh. For $s \in (1/2,1)$, the parameter $\sigma$ corresponds to $1$. Meanwhile, for $s \in (1/4,1/2)$, we have that $\sigma = 2s$. Consequently, for $s \in (1/4,1)$ our proposed scheme over-perform the one developed in \cite{AO}; see Section \ref{sec:numerics} for an illustration. For $s \in (0,1/4]$, we have that $\sigma = 1/2$ and, for these values of $s$, we derive a linear order of convergence for the proposed scheme.}

The outline of this paper is as follows. In Section \ref{sec:notation} we recall the definition of the fractional powers of elliptic operators via spectral theory and introduce some suitable functional framework. In Section \ref{sec:control} we define the fractional and extended optimal control problems. Section \ref{subsec:control} presents regularity results for the optimal control variables, in both H\"older and Sobolev spaces. Section \ref{sec:truncated} is advocated to the \emph{truncated optimal control problem}. We also state the approximation \EO{and regularity} properties of its solution. In Section \ref{sec:state_equation} we review the a priori error analysis of \cite{NOS} for the state equation \eqref{eq:alpha_harm} \EO{and in Section \ref{subsec:apriori_fractional} we extend these results to a scenario where $\Omega$ is a $C^2$ curved domain.} In Section \ref{sec:fd} we propose a fully--discrete scheme for the \emph{fractional optimal control problem} that discretizes the optimal control with piecewise linear functions. For $s \in (0,1)$, we provide an a priori error analysis for the proposed scheme. The analysis is valid in any dimension. Finally, in Section \ref{sec:numerics}, we present numerical experiments that illustrate the a priori error analysis of Section \ref{sec:fd}.

\section{Notation and preliminaries}
\label{sec:notation}

\subsection{Notation}
\label{subsec:notation}
Throughout this work $\Omega$ is an open, bounded and connected domain of $\R^n$ ($n\geq1$) with Lipschitz boundary $\partial\Omega$. We will follow the notation of \cite{AO,NOS3} and define the semi--infinite cylinder with base $\Omega$ and its lateral boundary, by $\C = \Omega\times (0,\infty)$ and $\partial_L \C = \partial \Omega \times [0,\infty)$, respectively. Given $\Y>0$, we define the truncated cylinder $ \C_\Y = \Omega \times (0,\Y)$ and its lateral boundary $\partial_L\C_\Y$ accordingly.

Since we will be dealing with objects defined in $\R^{n+1}$ and the extended $n+1$--dimension will play a special role, we denote a vector $x\in \R^{n+1}$ by
\[
  x =  (x_1,\ldots,x_n, x_{n+1}) = (x', x_{n+1}) = (x',y),
\]
with $x_i \in \R$ for $i=1,\ldots,{n+1}$, $x' \in \R^n$ and $y\in\R$.

The matrix of coefficients $A(x')$, that defines the operator $\mathcal{L}$ in \eqref{eq:L}, is measurable and bounded in $\Omega$, and uniformly elliptic. We denote by $\calLs$, with $s \in (0,1)$, the fractional powers of the operator $\mathcal{L}$. The parameter $\alpha$ belongs to $(-1,1)$ and is related to the power $s$ of the fractional operator $\calLs$ by the formula $\alpha = 1 -2s$.

Finally, the relation $a \lesssim b$ indicates that $a \leq Cb$, with a constant $C$ that depends neither on $a$ or $b$ nor the discretization parameters. The value of $C$ might change at each occurrence.

\subsection{Fractional powers of a second order elliptic operator}
\label{subsec:frac_laplacian}
We consider the definition based on spectral theory \cite{Pablo,CDDS:11,NOS}. The operator $\mathcal{L}^{-1}:L^2(\Omega) \rightarrow L^2(\Omega)$, that solves $\mathcal{L}w = f$ in $\Omega$ and $w=0$ on $\partial \Omega$, is compact, symmetric, and positive, so its spectrum $\{ \lambda_k^{-1} \}_{k\in \mathbb N}$ is discrete, real, and positive and accumulates at zero. Moreover, the eigenfunctions 
\begin{equation}
\label{eq:eigenfunctions}
 \{ \varphi_k \}_{k \in \mathbb{N}}: \qquad \mathcal{L} \varphi_k = \lambda_k \varphi_k \textrm{ in } \Omega, \qquad \varphi_k = 0 \textrm{ on } \partial\Omega, \qquad k \in \mathbb{N}, 
\end{equation}
form an orthonormal basis of $L^2(\Omega)$. With this spectral decomposition at hand, the fractional powers of $\mathcal{L}$ can be defined, for $w \in C_0^{\infty}(\Omega)$, by
\begin{equation}
  \label{def:second_frac}
  \calLs w  = \sum_{k=1}^\infty \lambda_k^{s} w_k \varphi_k,
\end{equation} 
where $w_k = \int_{\Omega} w \varphi_k $ and $s \in (0,1)$. By density $\calLs$ can be extended to the space
\begin{equation}
\label{def:Hs}
  \Hs = \left\{ w = \sum_{k=1}^\infty w_k \varphi_k: 
  \sum_{k=1}^{\infty} \lambda_k^s w_k^2 < \infty \right\},
\end{equation}
which, as a consequence of the theory of Hilbert scales \cite{Lions}, coincides with the space $[H_0^1(\Omega),L^2(\Omega) ]_{1-s}$. \EO{If $s\in(1,2]$, and $\Omega$ is, for instance, an open bounded and convex domain, we have that $\Hs = H^s(\Omega)\cap H^1_0(\Omega)$; see \cite{ShinChan}.} For $ s \in (0,1)$ we denote by $\Hsd$ the dual space of $\Hs$.

\subsection{The nonuniformly extension problem}
\label{subsec:CaffarelliSilvestre}
The localization results of Caffarelli and Silvestre \cite{BCdPS:12,CS:07,CDDS:11} require us to deal with the \emph{nonuniformly} elliptic problem \eqref{eq:alpha_harm} and consequently with Lebesgue and Sobolev spaces with the weight $y^{\alpha}$ for $\alpha \in (-1,1)$. Following \cite{Javier,Turesson}, we define, for an open set $D \subset \R^{n+1}$, the weighted Lebesgue space 
\[
L^2(|y|^\alpha,D) = \left\{w \in L^1_{\mathrm{loc}}(D): \| w \|_{L^2(|y|^{\alpha},D)}^2 = \int_{D}|y|^{\alpha} w^2 < \infty \right\},
\]
and the weighted Sobolev space
\[
H^1(|y|^{\alpha},D) =
  \left\{ w \in L^2(|y|^{\alpha},D): \| w \|_{H^1(|y|^{\alpha},D)} < \infty \right\},
\]
where 
\begin{equation}
\label{wH1norm}
\| w \|_{H^1(|y|^{\alpha},D)} =
\big( \| w \|^2_{L^2(|y|^{\alpha},D)} + \| \nabla w \|^2_{L^2(|y|^{\alpha},D)} \big)^{\frac{1}{2}}.
\end{equation}

Since $\alpha = 1-2s \in (-1,1)$, we have that $|y|^\alpha$ belongs to the Muckenhoupt class $A_2(\R^{n+1})$ \cite{Javier,FKS:82,GU,Muckenhoupt,Turesson}. Consequently, $H^1(|y|^{\alpha},D)$ endowed with the norm \eqref{wH1norm} is Hilbert and $C^{\infty}(D) \cap H^1(|y|^{\alpha},D)$ is dense in $H^1(|y|^{\alpha},D)$ (cf.~\cite[Proposition 2.1.2, Corollary 2.1.6]{Turesson}, \cite{KO84} and \cite[Theorem~1]{GU}). For convenience, we recall the definition of the Muckenhoupt class $A_2$ \cite{Muckenhoupt,Turesson}.

\begin{definition}[Muckenhoupt class $A_2$]
 \label{def:Muckenhoupt}
Let $\omega$ be a weight, that is $\omega \in L^1_{\mathrm{loc}}(\R^n)$ and $\omega >0$ a.e. in $\R^n$ with $n \geq 1$. We say that $\omega \in A_2(\R^n)$ if
\begin{equation}
\label{eq:Muckenhoupt}
  C_{2,\omega} = \sup_{B} \left( \fint_{B} \omega \right)
            \left( \fint_{B} \omega^{-1} \right) < \infty,
\end{equation}
where the supremum is taken over all balls $B$ in $\R^n$. 
\end{definition}

We now define the suitable weighted Sobolev space to analyze problem \eqref{eq:alpha_harm}:
\begin{equation*}
\HL(y^{\alpha},\C) = \left\{ w \in H^1(y^\alpha,\C): w = 0 \textrm{ on } \partial_L \C\right\}.
\end{equation*}
On this space, we have the following \emph{weighted Poincar\'e inequality} 
\begin{equation}
\label{Poincare_ineq}
\| w \|_{L^2(y^{\alpha},\C)} \lesssim \| \nabla w \|_{L^2(y^{\alpha},\C)}
\quad \forall w \in \HL(y^{\alpha},\C);
\end{equation}
see \cite[inequality (2.21)]{NOS}. Consequently, the seminorm on $\HL(y^{\alpha},\C)$ is equivalent to the norm \eqref{wH1norm}. For $w \in H^1(y^{\alpha},\C)$, $\tr w$ denotes its trace onto $\Omega \times \{ 0 \}$. We recall that, for $\alpha = 1-2s$, \cite[Proposition 2.5]{NOS} yields
\begin{equation}
\label{Trace_estimate}
\tr \HL(y^\alpha,\C) = \Hs,
\qquad
  \|\tr w\|_{\Hs} \leq C_{\tr} \| w \|_{\HLn(y^\alpha,\C)}.
\end{equation}

The weak formulation of problem \eqref{eq:alpha_harm} reads as follows: Find $\ue \in \HL(y^{\alpha},\C)$ such that
\begin{equation}
\label{eq:alpha_harm_weak}
 a(\ue ,\phi) =  \langle \zsf , \tr \phi \rangle_{\Hsd \times \Hs} \quad \forall \phi \in \HL(y^{\alpha},\C),
\end{equation}
where, for $w, \phi \in \HL(y^{\alpha},\C)$, the bilinear form $a$ is defined by
\begin{equation}
\label{eq:a}
a(w,\phi) =  \frac{1}{d_s} \int_{\C} y^{\alpha} \mathbf{A}(x) \nabla w  \cdot \nabla \phi,
\end{equation}
and $\langle \cdot, \cdot \rangle_{\Hsd \times \Hs}$ denotes the duality pairing between $\Hs$ and $\Hsd$ which, as a consequence of \eqref{Trace_estimate}, is well defined. We recall that the matrix of coefficients $\mathbf{A}(x',y) =  \textrm{diag} \{A(x'),1\}$ for all $(x',y) \in \C$. 

The conditions on $A$ and \eqref{Poincare_ineq} imply that $a$ is bounded and coercive in $\HL(y^\alpha, \C)$. Consequently, the well--posedness of  \eqref{eq:alpha_harm_weak} follows from the Lax--Milgram Lemma. In the manuscript, we will use repeatedly that $a(w,w)^{1/2}$ is a norm equivalent to \eqref{wH1norm}. We present the following estimate for problem \eqref{eq:alpha_harm_weak} \cite[Proposition 2.1]{CDDS:11}:
\begin{equation}
\label{eq:estimate_s}
  \| \nabla \ue \|_{L^2(y^{\alpha},\C)} \lesssim \| \usf \|_{\Hs} \lesssim \| \zsf \|_{\Hsd}.
\end{equation}

\EO{Define the \emph{Dirichlet-to-Neumann} operator $\textrm{N} : \Hs \to \Hsd$ by
\[
   \usf \in \Hs \longmapsto
   \textrm{N}(\usf) = \partial_{\nu}^{\alpha} \ue \in \Hsd.
\]  
We thus present the Caffarelli--Silvestre extension result \cite{CS:07,CDDS:11,ST:10}.}

\begin{theorem}[Caffarelli--Silvestre extension result]
\label{TH:CS}
If $\ue \in \HL(y^{\alpha},\C)$ and $u \in \Hs$ solve \eqref{eq:alpha_harm_weak} and \eqref{eq:fractional} for $\zsf \in \Hsd$, respectively, then $\usf = \tr \ue$ and 
\begin{equation*}
d_s \calLs \usf \EO{= \textrm{N}(\usf) =} \partial_{\nu}^{\alpha} \ue \text{ in } \Omega.
\end{equation*}
\end{theorem}

\section{The fractional and extended optimal control problems}
\label{sec:control}
For $J$ defined in \eqref{def:J}, the \emph{fractional optimal control problem} reads as follows: Find
$
\text{min } J(\usf,\zsf)
$
subject to the fractional state equation \eqref{eq:fractional} and the control constraints \eqref{eq:cc}. The set of \textit{admissible controls} $\Zad$ is defined by
\begin{equation*}
 \Zad= \{ \wsf \in L^2(\Omega): \asf \leq \wsf(x') \leq \bsf, \textrm{~~~a.e.~~~}  x' \in \Omega \},
\end{equation*}
where $\asf,\bsf \in \mathbb{R}$ and satisfy $\asf < \bsf$. The desired state $\usf_d \in L^2(\Omega)$ and $\vartheta >0$.

We define the fractional control--to--state operator map $\mathbf{S}: \Hsd \ni \zsf \mapsto \usf \in \Hs$, where $\usf=\usf(\zsf)$ solves \eqref{eq:fractional}. The operator $\mathbf{S}$ is linear and, as a consequence of \eqref{eq:estimate_s}, bounded from $\Hsd$ into $\Hs$. Given a control $\zsf \in \Hsd$, we define the fractional adjoint state $\psf = \psf(\zsf) \in \Hs$ as $\psf = \mathbf{S}(\usf - \usf_d)$.  With these elements at hand, we recall the following result \cite[Theorem 3.4]{AO}.

\begin{theorem}[existence, uniqueness and optimality conditions]
The optimal control problem \eqref{eq:J}--\eqref{eq:cc} has a unique optimal solution $(\ousf, \ozsf)  $ 
$\in$ $ \Hs \times \Zad$. The optimality conditions 
$\ousf = \mathbf{S} \bar{\zsf} \in \Hs$, $\opsf = \mathbf{S} ( \ousf - \usf_{\textrm{d} } ) \in \Hs$, and
\begin{equation}
 \label{vi_fractional}
 \ozsf \in \Zad, \qquad(\vartheta \ozsf + \opsf, \zsf - \ozsf )_{L^2(\Omega)} \geq 0 \quad \forall \zsf \in \Zad
\end{equation}
hold. These conditions are necessary and sufficient. 
\label{TH:fractional_control}
\end{theorem}

Since the fractional powers of $\mathcal{L}$ are non--local operators, the design and analysis of solution techniques for problem \eqref{eq:J}--\eqref{eq:cc} is far for being trivial. To circumvent such a difficulty, in reference \cite{AO} the authors invoke the localization results stated in Theorem \ref{TH:CS} and propose the following equivalent optimal control problem: Find
$
\text{min } J(\tr \ue,\zsf)
$
subject to the \emph{linear} and \emph{extended} state equation
\begin{equation}
\label{eq:extension_weak}
  a(\ue,\phi) = \langle \zsf, \tr \phi \rangle_{\Hsd \times \Hs}
    \quad \forall \phi \in \HL(y^{\alpha},\C)
\end{equation}
and the control constraints 
$
 \zsf \in \Zad;
$
the bilinear form $a$ is defined as in \eqref{eq:a}. We will refer to this problem as the \emph{extended optimal control problem}. We recall that the trace operator $\tr$ is defined in Section \ref{subsec:CaffarelliSilvestre}. The equivalence of the aforementioned optimal control problems follows from \cite[Theorem 3.12]{AO}: $\ousf(\ozsf) = \tr \oue(\ozsf)$.

We define the extended control--to--state operator $\mathbf{G}: \Hsd \ni \zsf  \mapsto \tr \ue \in \Hs$, where $\ue = \ue(\zsf)$ solves \eqref{eq:extension_weak}. Invoking Theorem \ref{TH:CS}, we conclude that the actions of $\mathbf{G}$ and $\mathbf{S}$ coincide and then that $\mathbf{G}$ is a linear and continuous mapping from $\Hsd$ into $\Hs$.

We also define the extended adjoint state $\pe = \pe(\zsf) \in \HL(y^{\alpha},\C)$, associated to the extended state $\ue = \ue(\zsf)$, as the unique solution to
\begin{equation}
\label{eq:extension_adjoint}
  a(\phi,\pe) = ( \tr \ue - \usfd, \tr \phi )_{L^2(\Omega)} \quad \forall \phi \in \HL(y^{\alpha},\C).
\end{equation}
With this definition at hand, we present the following result \cite[Theorem 3.11]{AO}.

\begin{theorem}[existence, uniqueness and optimality system]
\label{TH:extended_control}
The extended optimal control problem has a unique optimal solution $(\oue, \ozsf)$. The optimality system
\begin{equation*}
\begin{dcases}
\oue = \oue(\ozsf) \in \HL(y^{\alpha},\C) \textrm{ solution to } \eqref{eq:extension_weak}, \\
\ope = \ope(\ozsf)  \in \HL(y^{\alpha},\C) \textrm{ solution to } \eqref{eq:extension_adjoint}, 
\\
\ozsf \in \Zad, \quad (\tr \bar{\pe} + \vartheta \ozsf , \zsf - \ozsf )_{L^2(\Omega)} \geq 0 \quad \forall \zsf \in \Zad,
\end{dcases}
\end{equation*}
holds. These conditions are necessary and sufficient. 
\end{theorem}

\EO{We conclude this section with a representation formula for $\ue$ that is based on the eigenpairs $\{ \lambda_k, \varphi_k\}$ defined in \eqref{eq:eigenfunctions}. Let $\usf(x')=\sum_k \usf_k\varphi_k(x')$ be the solution to \eqref{eq:fractional}, then the solution $\ue$ to problem \eqref{eq:extension_weak} can be written as follows:
\begin{equation}
\label{eq:representation_formula}
  \ue(x,t) = \sum_{k=1}^\infty \usf_k \varphi_k(x') \psi_k(y),
\end{equation}
where, for $k \in \mathbb{N},$ $\psi_k$ solves
\begin{equation}
\label{psik}
\psi_k'' + \alpha y^{-1}\psi_k' - \lambda_k \psi_k = 0, \quad \psi_k(0) = 1, \quad \psi_k(y) \to 0 \mbox{ as } y \to \infty.
\end{equation}
If $s=1/2$, then $\psi_k(y) = e^{-\sqrt{\lambda_k}y}$. If $s \in (0,1)\setminus\{1/2\}$ and $c_s = 2^{1-s}/\Gamma(s)$, then
\[
  \psi_k(y) = c_s \left(\sqrt{\lambda_k}y\right)^s K_s(\sqrt{\lambda_k} y),
\]
where $K_s$ is the modified Bessel function of the second kind \cite[Chapter~9.6]{Abra}. We refer the reader to \cite[Section 2.4]{NOS} and \cite[Proposition~2.1]{CDDS:11} for details.}

\subsection{Regularity of the fractional optimal control}
\label{subsec:control}

Since we will be concerned with the approximation of the solution to the extended optimal control problem, it is essential to study the regularity properties of $(\bar{\usf},\bar{\psf},\bar{\zsf})$. To derive such results, we will assume different conditions on the smoothness of the domain $\Omega$ and the matrix of coefficients $A$ defining $\mathcal{L}$ in \eqref{eq:L}; we will be precise regarding these assumptions when is needed. We start by recalling the results of Lemma 3.5 in \cite{AO}.

\begin{lemma}[regularity of the optimal control]
\label{le:reg_control}
Let $\ozsf \in \Zad$ be the fractional optimal control. If $\Omega$ is convex, $A \in C^{0,1}(\bar \Omega)$, $\usfd \in \mathbb{H}^{1-s}(\Omega)$ and $\asf < 0 < \bsf$, then $\ozsf \in H_0^1(\Omega)$.
\end{lemma}

 We recall the so--called \emph{projection formula} for the optimal control $\bar \zsf$. If $\vartheta >0$ and $\opsf$ denotes the fractional optimal adjoint state, then the projection formula
\begin{equation}
\label{projection_formula}
    \ozsf(x') = \textrm{proj}_{[\asf,\bsf]} \left(- \frac{1}{\vartheta} \opsf(x') \right)
\end{equation}
is equivalent to \eqref{vi_fractional}, where $\textrm{proj}_{[\asf,\bsf]}(v) = \min\{ \bsf, \max \{ \asf,v \} \}$; see \cite[Section 2.8]{Tbook}. 
\EO{We comment on the assumption $\asf < 0 < \bsf$ of Lemma \ref{le:reg_control}: the bootstrap argument developed in the proof of Lemma 3.5 in \cite{AO} relies on the projection formula \eqref{projection_formula} and thus requires the assumption $\asf < 0 < \bsf$ in order to preserve the boundary values of $\ozsf \in \mathbb{H}^s(\Omega)$. We refer the reader to the proof of \cite[Lemma 3.5]{AO} for details.} 

On the basis of the projection formula \eqref{projection_formula} we derive the following regularity properties, on Sobolev spaces, for $\bar \zsf$ and $\bar \psf$.

\begin{lemma}[regularity of the optimal state and adjoint state]
\label{le:reg_state}
\EO{If $\usfd \in \mathbb{H}^{1-s}(\Omega)$, $\Omega$ is convex, $\asf < 0 < \bsf$ and $A \in C^{0,1}(\bar \Omega)$, then $\ousf \in \mathbb{H}^{\kappa}(\Omega)$ for $\kappa = \min \{1+2s,2\}$. In addition, if $\usfd \in H_0^1(\Omega)$, then $\opsf \in \mathbb{H}^{\kappa}(\Omega)$ for $\kappa = \min \{1+2s,2\}$.}
\end{lemma}
\begin{proof}
\EO{An application of Lemma \ref{le:reg_control} yields $\ozsf \in H_0^1(\Omega)$. Then, since $\Omega$ is convex, we invoke the fact that $\mathcal{L}^s$ is a pseudodifferential operator of order $2s$ and Theorem 3.2.1.2 in \cite{Grisvard}, to conclude that $\bar{\usf}$, the solution to problem \eqref{eq:fractional}, belongs to $\mathbb{H}^{\kappa}(\Omega)$, where $\kappa = \min \{ 1+2s,2\}$. If $\usf_d \in H_0^1(\Omega)$, we have that $\bar \usf - \usf_d \in H_0^1(\Omega)$ and then a similar argument shows that the optimal adjoint state $\opsf \in \mathbb{H}^{\kappa}(\Omega)$. }
\end{proof}

In Section \ref{sec:fd} we will propose a numerical scheme to approximate the solution to \eqref{eq:J}--\eqref{eq:cc} and also provide an priori error analysis. The latter is based on H\"older--regularity results for the optimal control $\ozsf$. In order to obtain such results, we invoke the recent boundary regularity estimates for the solution $\usf$ to problem \eqref{eq:fractional}, derived by Caffarelli and Stinga in \cite{Pablo}, that assume some suitable smoothness properties on $\zsf$, the domain $\Omega$ and the matrix $A$. Since, in our setting, we have that $\zsf \in \Zad \subset L^{\infty}(\Omega)$, these results can be adapted to the following regularity estimates; see also \cite[Lemma 2.10]{CDDS:11} and \cite[Proposition 5.2]{Barrios20126133}.  

\begin{lemma}[boundary regularity for $\usf$]
\label{le:Pablo}
Let $\usf$ be the solution to \eqref{eq:fractional} with $\zsf \in L^{\infty}(\Omega)$. If $s \in (0,\tfrac{1}{2})$, $\Omega$ is a $C^1$ domain and $A \in C(\bar \Omega)$, then $\usf \in C^{0,2s}(\bar{\Omega})$, and 
\begin{equation}
\label{eq:bound_2s}
[\usf]_{C^{0,2s}(\bar \Omega)} \lesssim \| \usf \|_{\Hs} + \| \zsf \|_{L^{\infty}(\Omega)}.
\end{equation}
On the other hand, if $s \in (\tfrac{1}{2},1)$, $\Omega$ is a $C^{1,2s-1}$ domain and $A \in C^{0,2s-1}(\bar{\Omega})$, then $\usf \in C^{1,2s-1}(\bar{\Omega})$ and 
\[
[\usf]_{C^{1,2s-1}(\bar \Omega)} \lesssim \| \usf \|_{\Hs} + \| \zsf \|_{L^{\infty}(\Omega)}.
\]
\end{lemma}

When $s=\tfrac{1}{2}$, it is not completely evident how to adapt the techniques developed in \cite{Pablo} to derive regularity results for the solution $\usf$ under the assumption that $\zsf \in L^{\infty}(\Omega)$. For this reason, we conjecture the following regularity property: If $s = \tfrac{1}{2}$, $\Omega$ is a $C^1$ domain, $A \in C(\bar{\Omega})$ and $\zsf \in L^{\infty}(\Omega)$, then 
\begin{equation}
\label{conjecture}
\usf \in C^{0,\theta}(\bar \Omega)
\end{equation}
for every $\theta < 1$. We now derive the following regularity properties for the fractional optimal state $\ousf$ and the optimal adjoint state $\opsf$. \EO{To accomplish this task, we define}
\begin{equation}
 \EO{\Lambda(\ousf,\opsf,\usf_d)= \| \ousf \|_{\Hs} + \| \opsf \|_{\Hs}+\| \ozsf \|_{L^{\infty}(\Omega)} + \| \usf_d \|_{L^{\infty}(\Omega)}.}
\end{equation}

\begin{theorem}[boundary regularity of $\bar{\usf}$ and $\bar{\psf}$]
\label{th:reg_u_p}
Let $\usf_d \in L^{\infty}(\Omega)$. 
If $s \in (0,\tfrac{1}{2})$, $\Omega$ is a $C^1$ domain and $A \in C(\bar \Omega)$, then both $\ousf$ and $\opsf$ belong to $C^{0,2s}(\bar{\Omega})$. In addition,
\begin{equation}
\label{eq:reg_u_p_2s}
 [\ousf]_{C^{0,2s}(\bar \Omega)} + [\opsf]_{C^{0,2s}(\bar \Omega)}   \lesssim \Lambda(\ousf,\opsf,\usf_d).
\end{equation}
On the other hand, if $s \in (\tfrac{1}{2},1)$, $\Omega$ is a $C^{1,2s-1}$ domain and $A \in C^{0,2s-1}(\bar{\Omega})$, then both $\ousf$ and $\opsf$ belong to $C^{1,2s-1}(\bar{\Omega})$. In addition, 
\begin{equation}
\label{eq:reg_u_p_12s}
[\ousf]_{C^{1,2s-1}(\bar \Omega)} + [\opsf]_{C^{1,2s-1}(\bar \Omega)} \lesssim \Lambda(\ousf,\opsf,\usf_d).
\end{equation}
In both inequalities the hidden constant is independent of the optimal variables.
\end{theorem}
\begin{proof}
Let $s \in (0,\tfrac{1}{2})$. The results of Lemma \ref{le:Pablo} imply that $\ousf \in C^{0,2s}(\bar \Omega)$ and that \eqref{eq:bound_2s} holds. In view of the fact that $\ousf - \usf_d \in L^{\infty}(\Omega)$, we apply, again, the results of Lemma \ref{le:Pablo} and conclude that $\opsf \in C^{0,2s}(\Omega)$, and 
\begin{equation*}
[\psf]_{C^{0,2s}(\bar \Omega)} \lesssim \| \opsf \|_{\Hs} + \| \ousf \|_{L^{\infty}(\Omega)} + \| \ousf_d \|_{L^{\infty}(\Omega)}.
\end{equation*}
This estimate, in view of \eqref{eq:bound_2s}, allows us to derive \eqref{eq:reg_u_p_2s}. Analogous arguments can be applied to obtain \eqref{eq:reg_u_p_12s} for $s \in (\tfrac{1}{2},1)$.
\end{proof}

We now present the following improved regularity result for the optimal control $\ozsf$, which shows that, for $s \in (\tfrac{1}{2},1)$, and under some suitable assumptions, $\ozsf \in C^{0,1}(\bar \Omega)$.

\begin{theorem}[boundary regularity of $\ozsf$]
\label{th:reg_controlII}
Let $\usf_d \in L^{\infty}(\Omega)$. If $s \in (0,\tfrac{1}{2})$, $\Omega$ is a $C^1$ domain and $A \in C(\bar \Omega)$, then $\ozsf \in C^{0,2s}(\bar \Omega)$ and  
\begin{equation*}
[\ozsf]_{C^{0,2s}(\bar \Omega)} \lesssim \| \opsf \|_{\Hs} + \| \ousf \|_{L^{\infty}(\Omega)} + \| \usf_d  \|_{L^{\infty}(\Omega)}.
\end{equation*}
On the other hand, if $s \in (\tfrac{1}{2},1)$, $\Omega$ is a $C^{1,2s-1}$ domain and  $A \in C^{0,2s-1}(\bar{\Omega})$, then $\ozsf \in C^{0,1}(\bar{\Omega})$ and
\begin{equation*}
[\ozsf]_{C^{0,1}(\bar \Omega)} \lesssim \| \opsf \|_{\Hs} + \| \ousf \|_{L^{\infty}(\Omega)} + \| \usf_d  \|_{L^{\infty}(\Omega)}.
\end{equation*}
\end{theorem}
\begin{proof}
The desired regularity results follow from Theorem \ref{th:reg_u_p} in conjunction with the fact that the projection formula \eqref{projection_formula} maps continuously
$C^{0,\sigma}(\bar \Omega)$ into $C^{0,\sigma}(\bar \Omega)$ for $\sigma \in (0,1]$.
\end{proof}

We conclude this section with a conjecture regarding the regularity of the optimal control $\ozsf$ when $s = \tfrac{1}{2}$. If $\usf_d \in L^{\infty}(\Omega)$, $\Omega$ is a $C^1$ domain and $A \in C(\bar \Omega)$, then
\begin{equation}
 \label{eq:reg_c_3}
\ozsf \in C^{0,\theta}(\bar \Omega)
\end{equation}
for every $\theta < 1$. 


\section{The truncated optimal control problem}
\label{sec:truncated}
The state equation \eqref{eq:extension_weak} of the extended optimal control problem is posed on the semi--infinite domain $\C = \Omega \times (0,\infty)$. Therefore, it cannot be directly approximated with finite--element--like techniques. However, since the solution to \eqref{eq:extension_weak} decays exponentially in $y$ \cite[Proposition 3.1]{NOS}, by truncating $\C$ to $\C_\Y = \Omega \times (0,\Y)$, for a suitable truncation parameter $\Y \geq 1$, and setting a homogeneous Dirichlet condition on $y = \Y$, we only incur in an exponentially small error in terms of $\Y$ \cite[Lemma 4.6]{AO}.  We briefly review the results of \cite[Section 4]{AO}. To do this, we define
\[
  \HL(y^{\alpha},\C_\Y) := \left\{ v \in H^1(y^\alpha,\C_\Y): v = 0 \text{ on }
    \partial_L \C_\Y \cup \Omega \times \{ \Y\} \right\}.
\]
Then, the \emph{truncated optimal control problem} reads as follows: Find
$
 \text{min } J(\tr v,\rsf)
$
subject to the truncated state equation
\begin{equation}
\label{eq:truncated_state}
 a_\Y(v,\phi) = \langle \rsf, \tr \phi \rangle_{\Hsd \times \Hs}\quad \forall \phi \in \HL(y^{\alpha},\C_\Y)
\end{equation}
and the control constraints 
$
 \rsf \in \Zad.
$
The bilinear form $a_{\Y}$ is defined by
\begin{equation}
\label{eq:a_Y}
a_\Y(w,\phi) = \frac{1}{d_s} \int_{\C_\Y} {y^{\alpha}\mathbf{A}(x)} \nabla w \cdot \nabla \phi  \quad \forall w,\phi \in \HL(y^{\alpha},\C_\Y).
\end{equation}

We define the truncated control--to--state operator $\mathbf{H}: \Hsd \ni \rsf  \mapsto \tr v \in \Hs$, where $v = v(\rsf)$ denotes the unique solution to \eqref{eq:truncated_state}. The map $\mathbf{H}$ is linear and continuous; see \cite[Proposition 2.1]{CDDS:11}. With this operator at hand, we define the reduced cost functional $j: \Zad \ni \rsf \rightarrow j(\rsf) \in \mathbb{R}$ by
\begin{equation}
 \label{eq:red_J}
 j(\rsf) = J(\rsf, \mathbf{H} \rsf ),
\end{equation}
which is continuous and convex. In addition, the quadratic structure of $j$ implies that $j''(\qsf)(\rsf,\rsf)$ does not depend on $\qsf$ and is positive definite, that is
\begin{equation*}
 j''(\qsf)(\rsf,\rsf) \geq \vartheta \| \rsf\|^2_{L^2(\Omega)} \quad \forall \rsf \in L^2(\Omega).
\end{equation*}

The truncated adjoint state $p= p(\rsf) \in \HL(y^{\alpha},\C_{\Y})$, associated with the extended state $v=v(\rsf)$, is defined as the unique solution to
\begin{equation}
\label{eq:truncated_adjoint}
a_\Y(\phi,p) = (\tr v - \usfd, \tr \phi )_{L^2(\Omega)} \qquad \forall \phi \in \HL(y^{\alpha},\C_{\Y}).
\end{equation}
We then have the following result \cite[Theorem 4.5]{AO}.

\begin{theorem}[existence, uniqueness and optimality system]
\label{TH:truncated_control}
The truncated optimal control problem has a unique solution $(\bar{v}, \orsf)$. The optimality system
\begin{equation}
\begin{dcases}
 \bar{v} = \bar{v}(\orsf) \in \HL(y^{\alpha},\C_{\Y}) \textrm{ solution to } \eqref{eq:truncated_state}, \\
 \bar{p} = \bar{p}(\orsf) \in \HL(y^{\alpha},\C_{\Y}) \textrm{ solution to } \eqref{eq:truncated_adjoint}, \\
 \orsf \in \Zad, \quad (\tr \bar{p} + \vartheta \orsf , \rsf - \orsf )_{L^2(\Omega)} \geq 0 \quad \forall \rsf \in \Zad,
\end{dcases}
\label{op_truncated} 
\end{equation}
holds. These conditions are necessary and sufficient. 
\end{theorem}

The following exponential approximation properties follow from \cite[Lemma 4.6]{AO}.

\begin{lemma}[exponential convergence]
\label{LE:exp_convergence}
If $(\oue,\ozsf) \in \HL(y^{\alpha},\C) \times \Hs$ and $(\bar v, \orsf) \in \HL(y^{\alpha},\C_{\Y}) \times \Hs$ solve the extended and truncated optimal control problems, respectively, then
\begin{align*}
  \| \bar \zsf - \bar \rsf \|_{L^2(\Omega)} & \lesssim e^{-\sqrt{\lambda_1} \Y/4} \left(\| \bar{\rsf} \|_{L^2(\Omega)} + \| \usf_d \|_{L^2(\Omega)} \right),\\
  \| \nabla \left( \bar{\ue}  - \bar{v}  \right) \|_{L^2(y^{\alpha},\C)} & \lesssim e^{-\sqrt{\lambda_1} \Y/4} \left(\| \bar{\rsf} \|_{L^2(\Omega)} + \| \usf_d \|_{L^2(\Omega)} \right),
\end{align*}
where $\lambda_1$ denotes the first eigenvalue of the operator $\mathcal{L}$.
\end{lemma}

We conclude this section with the following regularity result for the truncated optimal control $\orsf$ and the truncated optimal adjoint state $\tr \bar{p}$.

\begin{proposition}[Sobolev--regularity of $\orsf$ and $\tr \bar{p}$] \rm
\label{pro:regularity_control_r_Sobolev}
\EO{Let $\orsf \in \Zad$ be the truncated optimal control, $\Omega$ be a convex domain and $A \in C^{0,1}(\bar \Omega)$. If $\asf < 0 < \bsf$ and $\usf_d \in \mathbb{H}^{1-s}(\Omega)$, then $\orsf \in H_0^{1}(\Omega)$. If, in addition, $\usf_d \in H_0^1(\Omega)$, then $\tr \bar{p} \in \mathbb{H}^{\kappa}(\Omega)$, where $\kappa = \min \{1+2s,2 \}$.} 
\end{proposition}
\begin{proof}
\EO{The techniques of \cite[Remark 25]{NOS3} allow us to transfer the regularity results of Lemmas \ref{le:reg_control} and \ref{le:reg_state} to $\orsf$ and $\tr \bar{p}$, respectively. To elucidate these results, we write, in view of separation of variables, a representation formula for the solution to \eqref{eq:truncated_state}: $v(x',y) = \sum_{k} v_k \varphi_k(x') \chi_k(y)$, where $\chi_k$ solves 
\begin{equation}
\label{chik}
\chi_k'' + \alpha y^{-1} \chi_k' - \lambda_k \chi_k = 0, \quad \chi_k(0) = 1, \quad \chi_k(\Y)= 0,
\end{equation}
and $\{ \varphi_k \}$ denote the eigenfunctions of the operator $\mathcal{L}$. If $I_s$ and $K_s$ denote the modified Bessel functions of first and second kind \cite[Section 9.6]{Abra}, then 
\[
  \chi_k(y) = \Big( \sqrt{\lambda_k} y \Big)^s \Big( a_{k,s} K_s( \sqrt{\lambda_k} y ) + b_{k,s} I_s( \sqrt{\lambda_k} y ) \Big). 
\] 
The arguments of \cite[Remark 25]{NOS3} thus reveal that $b_{k,s}= -c_sK_s(\sqrt{\lambda_k}\Y)I_s(\sqrt{\lambda_k}\Y)^{-1}$ and that $a_{k,s}=c_s = 2^{1-s}/\Gamma(s)$. Define $e_{k,s}=2^{1-s}b_{k,s}/\Gamma(s)$. Then, in view of the properties of the Bessel functions we conclude that $\{ e_{k,s} \}$ decays exponentially to $0$ as $k \uparrow \infty$, and in addition, that
\[
 -\lim_{y \downarrow 0} y^{\alpha} v_y(x',y) = \sum_{k=1}^{\infty} \lambda_k^s (d_s - e_{k,s}) v_k \varphi(x'),
\]
where $d_s= 2^{\alpha}\Gamma(s)/\Gamma(1-s)$. Since $\bar{v}$ solves problem \eqref{eq:truncated_state}, these arguments show that $\tr \bar{v}$ solves the following fractional PDE:
\begin{equation}
\label{eq:Ls_exp}
 \mathfrak{L}^s \tr \bar{v} = \orsf,
\end{equation}
where, for $w \in C_0^{\infty}(\Omega)$ and $s \in (0,1)$, the fractional operator $\mathfrak{L}^s$ is defined by
\[
 \mathfrak{L}^s w = \sum_{k=1}^{\infty} \lambda_k^s w_k (1 - e_{k,s}d_s^{-1});
\]
compare with \eqref{def:second_frac}. Since $\{ e_{k,s} \}$ decays exponentially to $0$ as $k \uparrow \infty$, in particular is uniformly bounded, and then, we can extend the operator $\mathfrak{L}^s$ to the space $\Hs$ defined by \eqref{def:Hs}. We can thus proceed, on the basis of \eqref{eq:Ls_exp}, as in the proof of Lemma 3.5 in \cite{AO} to derive that $\orsf \in H_0^1(\Omega)$. In view of this result, the arguments of Lemma \ref{le:reg_state} can be applied to derive that $\tr \bar{p} \in \mathbb{H}^{\kappa}(\Omega)$ for $\kappa = \min \{1+2s,s \}$. For brevity we skip the details.}
\end{proof}

\EO{The regularity properties derived by Caffarelli and Stinga in \cite{Pablo} for $\usf$, the solution to \eqref{eq:fractional}, are based on the the fact that $\usf = \tr \ue$ and that $\ue$ solves \eqref{eq:extension_weak}; see the proof of \cite[Theorem 6.1]{Pablo}. In fact, since these arguments rely on the sctructure of the operator involved in the extension problem and are local in the extended dimension, they can also be applied to derive regularity properties for $\tr v$, where $v$ solves \eqref{eq:truncated_state}. This, on the basis of Theorem \ref{th:reg_controlII}, provides the following H\"older regularity results for the truncated optimal control $\orsf$.}

\begin{proposition}[H\"older--regularity of $\orsf$] \rm
\label{pro:regularity_control_r_Holder}
\EO{Let $\orsf \in \Zad$ be the truncated optimal control. Let $\usf_d \in L^{\infty}(\Omega)$, $\Omega$ be a $C^1$ domain and $A \in C(\bar \Omega)$. In addition, if for $s \in (\tfrac{1}{2},1)$, we have that $\Omega$ is a $C^{1,2s-1}$ domain and $A \in C^{0,2s-1}(\bar \Omega)$, then  
\begin{equation}
\label{eq:reg_r}
 \orsf \in C^{0.2s}(\bar \Omega) \textrm{ for } s \in \left(0,\frac{1}{2}\right), \quad \textrm{and} \quad  \orsf \in C^{0,1}(\bar \Omega) \textrm{ for } s \in \left(\frac{1}{2},1\right).
\end{equation}}
\end{proposition}

\EO{For $s = \tfrac{1}{2}$, we conjecture the following regularity result. If $\usf_d \in L^{\infty}(\Omega)$, $\Omega$ is a $C^1$ domain and $A \in C(\bar \Omega)$, then
\begin{equation}
 \label{eq:reg_c_r}
\orsf \in C^{0,\theta}(\bar \Omega)
\end{equation}
for every $\theta < 1$.}
\section{A finite element method for the state equation}
\label{sec:state_equation}
In the next section we will propose a fully discrete scheme to approximate the solution to the optimal control problem \eqref{eq:J}--\eqref{eq:cc}. The analysis relies, first, on the localization results of Section \ref{sec:control}, and second, \EO{on finite element approximation techniques for solving  \eqref{eq:truncated_state} on \emph{curved domains}; the latter being an extension of the priori error analysis developed in \cite{NOS}. We comment that such an analysis is not trivial, since involves anisotropic meshes in the extended dimension and the nonuniform weight $y^{\alpha}$ ($\alpha = 1-2s \in (-1,1)$), which degenerates $(s < 1/2)$ or blows up $(s > 1/2)$.}

It is instructive to review the results of \cite{NOS}, which assume that $\Omega$ is a convex polytopal subset of $\R^n$ ($n\geq1$) with boundary $\partial \Omega$. To do this, we start by recalling the 
regularity properties of $\ue$ and $v$, solutions to \eqref{eq:extension_weak} and \eqref{eq:truncated_state}, respectively.  The second order regularity of $\ue$ is much worse in the extended direction. In fact \cite[Theorem 2.7]{NOS} (see \cite[Remark 25]{NOS3} for $v$) yields
\begin{align}
    \label{reginx}
  \| \Delta_{x'} \ue\|_{L^2(y^{\alpha},\C)} + 
  \| \partial_y \nabla_{x'} \ue \|_{L^2(y^{\alpha},\C)}
  & \lesssim \| f \|_{\Ws}, \\
\label{reginy}
  \| \ue_{yy} \|_{L^2(y^{\beta},\C)} &\lesssim \| f \|_{L^2(\Omega)},
\end{align}
with $\beta > 2\alpha + 1$. These regularity estimates have important consequences in the design of efficient numerical techniques to solve \eqref{eq:truncated_state}; they suggest that \emph{graded} meshes in the extended $(n+1)$--dimension must be used. We recall the construction of the family of meshes $\{ \T_{\Y} \}$ over $\C_{\Y}$ used in \cite{AO,NOS3}. First, we consider a partition $\mathcal{I}_\Y$ of the interval $[0,\Y]$ with mesh points
\begin{equation}
\label{eq:graded_mesh}
y_k = k^\gamma M^{-\gamma} \Y, \quad k=0,\dots,M,
\end{equation}
where $\gamma > 3/(1-\alpha)=3/(2s)>1$. Second, we consider $\T_{\Omega} = \{K\}$ to be a conforming mesh of $\Omega$, where $K \subset \R^n$ is an element that is isoparametrically equivalent either to the unit cube $[0,1]^n$ or the unit simplex in $\R^n$. We denote by $\Tr_{\Omega}$ the collection of all conforming refinements of an original mesh $\T_{\Omega}^0$. We assume that $\Tr_{\Omega}$ is shape regular \cite{MR0520174}. We then construct a mesh $\T_{\Y}$ over $\C_{\Y}$ as the tensor product triangulation of $\T_{\Omega} \in \Tr_{\Omega}$ and $\mathcal{I}_{\Y}$. We denote by $\Tr$ the set of all the meshes obtained with this procedure, and recall that $\Tr$ satisfies the following weak shape regularity condition: If $T_1 = K_1 \times I_1$ and $T_2=K_2\times I_2 \in \T_\Y$ have nonempty intersection, then there exists a positive constant $\sigma_{\Y}$ such that
\begin{equation}
\label{eq:weak_shape_reg}
     h_{I_1} h_{I_2}^{-1} \leq \sigma_{\Y},
\end{equation}
where $h_I = |I|$. This weak shape regularity condition allows for anisotropy in the extended variable $y$ \cite{DL:05,NOS,NOS2}.

Given a mesh $\T_{\Omega} \in \Tr_\Omega$, we denote by $\N(\T_{\Omega})$ the set of all its interior nodes. We define $h_{\T_{\Omega}} = \max_{K \in \T_{\Omega}} h_K$. For $\T_{\Y} \in \Tr$, we define the finite element space 
\begin{equation}\
\label{eq:FESpace}
  \V(\T_\Y) = \left\{
            W \in C^0( \bar{\C}_\Y): W|_T \in \mathcal{P}_1(K) \otimes \mathbb{P}_1(I) \ \forall T \in \T_\Y, \
            W|_{\Gamma_D} = 0
          \right\},
\end{equation}
where $\Gamma_D = \partial_L \C_{\Y} \cup \Omega \times \{ \Y\}$ is the Dirichlet boundary. The set $\mathcal{P}_1(K)$ is $\mathbb{P}_1(K)$ -- the space of polynomials of total degree at most $1$ -- when the base $K$ of an element $T = K \times I$ is a simplex. If $K$ is a cube, $\mathcal{P}_1(K)$ stand for $\mathbb{Q}_1(K)$ -- the space of polynomials of degree not larger than $1$ in each variable. We also define the finite element space $\U(\T_{\Omega})=\tr \V(\T_{\Y})$. We assume that $\# \T_{\Omega} \approx M^n$. This, in view of the fact that $\#\T_{\Y} = M \, \# \T_{\Omega}$, implies that $\#\T_\Y \approx M^{n+1}$.

The Galerkin approximation of \eqref{eq:truncated_state} is the function $V \in \V(\T_{\Y})$ that satisfies
\begin{equation}
\label{eq:discrete_state}
  a_{\Y}(V,W) = (\rsf, \textrm{tr}_{\Omega} W )_{L^2(\Omega)}
  \quad \forall W \in \V(\T_{\Y}),
\end{equation}
where $a_{\Y}$ is defined in \eqref{eq:a_Y}.  We present the a priori error estimates derived in \cite[Theorem 5.4]{NOS} and \cite[Corollary 7.11]{NOS}.

\begin{theorem}[a priori error estimates]
\label{TH:fl_error_estimates}
Let $\T_\Y \in \Tr$ and $\V(\T_\Y)$ be defined by \eqref{eq:FESpace}. If $\ue(\rsf) \in \HL(y^{\alpha},\C)$ solves \eqref{eq:discrete_state} with $\zsf$ replaced by $\rsf$, then
\begin{equation}
\label{eq:fl_optimal_rate}
  \| \nabla( \ue(\rsf) - V) \|_{L^2(y^\alpha,\C)} \lesssim
|\log(\# \T_{\Y})|^s(\# \T_{\Y})^{-1/(n+1)} \|\rsf \|_{\mathbb{H}^{1-s}(\Omega)},
\end{equation}
where $\Y \approx \log(\# \T_{\Y})$. Alternatively, if $\usf(\rsf)$ denotes the solution to \eqref{eq:fractional} with forcing term $\rsf$, then
\begin{equation*}
\| \usf(\rsf) - \tr V \|_{\Hs} \lesssim
|\log(\# \T_{\Y})|^s(\# \T_{\Y})^{-1/(n+1)} \|\rsf \|_{\mathbb{H}^{1-s}(\Omega)}.
\end{equation*}
\end{theorem}

\begin{remark}[domain and data regularity]
\label{rm:dom_and_data2}
The results of Theorem~\ref{TH:fl_error_estimates} hold only if $\rsf \in \mathbb{H}^{1-s}(\Omega)$ and the domain $\Omega$ is sufficiently regular, for instance, convex.
\end{remark}

\subsection{A priori error analysis for fractional diffusion on curved domains}
\label{subsec:apriori_fractional}

In order to guarantee the regularity results of Theorem \ref{th:reg_controlII} and Proposition \ref{pro:regularity_control_r_Holder} we need the following smoothness assumptions on the domain $\Omega$ and the matrix $A$ that defines the operator $\mathcal{L}$ in \eqref{eq:L}:
\begin{enumerate}[(a)]
 \item \label{eq:assump_1} For $s \in (0,\tfrac{1}{2})$, $\Omega$ is a convex $C^1$ domain and $A \in C(\bar \Omega)$.
 \item \label{eq:assump_2} For $s \in (\tfrac{1}{2},1)$, $\Omega$ is a convex $C^{1,2s-1}$ domain and $A \in C^{0,2s-1}(\bar \Omega)$.
\end{enumerate}
For the critical case $s = 1/2$, we assume \eqref{eq:reg_c_r}, which in turns requires that $\Omega$ is a convex $C^1$ domain and $A \in C(\bar \Omega)$. Since we will be working on the basis of assumptions \eqref{eq:assump_1} and \eqref{eq:assump_2} we cannot consider the domain $\Omega$ to be a convex polytopal domain in $\mathbb{R}^n$ $(n \geq 1)$. Instead, we consider a family of open, bounded and convex polytopal domains $\{ \Omega_{\T} \} $, based on a family of shape regular triangulations $\{ \T_{\Omega} \}$, made of simplices, that approximate $\Omega$ in the following sense:  
\begin{equation}
\label{eq:Omega_T}
 \N(\T_{\Omega}) \subset \bar \Omega_{\T}, \quad \N(\T_{\Omega}) \cap \partial \Omega_{\T} \subset \partial \Omega, \quad
 |\Omega\setminus\Omega_\T| \lesssim h_{\T_{\Omega}}^2,
\end{equation}
where $\N(\T_{\Omega})$ denotes the set of all the nodes of the mesh $\T_{\Omega}$; we refer the reader to \cite{MR1115237} or \cite[Chapter 5.2]{MR773854} for details; see also \cite{MR1972729}. \EO{From now on, we assume $\Omega$ to be a convex $C^2$ domain; the convexity property implies that $\{ \Omega_\T \} \subset \Omega$.}

\begin{remark}[previous results and regularity of $\Omega$]
\EO{
 We remark that, since the error estimates derived in \cite{AO} are based on the $H^1$--regularity of the optimal control $\ozsf$, the previous construction of the sequence $\{ \Omega_{\T} \}$ is not needed in \cite{AO}: $\Omega$ can be taken as a convex polytopal domain in $\R^n$ (see Lemma \ref{le:reg_control}). In contrast, we will operate under the regularity results of Theorem \ref{th:reg_controlII} and Proposition \ref{pro:regularity_control_r_Holder} and therefore, in order to have the validity of such results and handle the curved domain, we assume $\Omega$ to be a convex $C^2$ domain.}
\end{remark}

On the basis of the previous construction, it is thus necessary to modify the definition of the finite element space $\V(\T_{\Y})$. For the sake of simplicity, we keep the notation and define
\begin{equation}
\label{eq:V_curved}
  \V(\T_\Y) = \left\{ W \in C^0( \bar{\C}_\T ): W|_T \in \mathbb{P}_1(K) \otimes \mathbb{P}_1(I) \, \forall T, \, W|_{ \bar{\Omega}\setminus \Omega_{\T} \times (0,\Y] } = 0 \right\},
\end{equation}
where $\C_{\T} = \Omega_{\T} \times (0,\Y)$, $T \in \T_{\Y}$ and the mesh $\T_{\Y}$ of $\C_\Y$ is constructed as the tensor product of $\T_{\Omega}$ and $\mathcal{I}_\Y$; the latter being defined in Section \ref{sec:state_equation}. The discrete state equation then reads: Find $V \in \V(\T_{\Y})$ such that
\begin{equation}
\label{fd_a_new}
a_{\T}(V,W) =  ( \rsf, \tr W )_{L^2(\Omega)} \quad \forall W \in \V(\T_{\Y}),
\end{equation}
where
\begin{equation}
\label{eq:a_T}
 a_{\T}(V,W) = \frac{1}{d_s} \int_{\C_{\T}} y^{\alpha} \mathbf{A}(x) \nabla V \cdot \nabla W.
\end{equation}

We now present an extension of the a priori error estimate \eqref{eq:fl_optimal_rate} of \cite[Theorem 5.4]{NOS}. In contrast to \eqref{eq:fl_optimal_rate} the derived estimate allows us to consider curved domains.
\begin{lemma}[energy--error estimate on curved domains]
\label{le:curved_domains_1}
\EO{Let $\ue(\rsf) \in \HL(y^{\alpha},\C)$ be the solution to \eqref{eq:extension_weak} with $\zsf$ replaced by $\rsf$, and let $V \in \mathbb{V}(\T_{\Y})$ be the solution to \eqref{fd_a_new}. If $\Omega$ is a convex $C^2$ domain, $A \in C^{0,1}(\bar \Omega)$, $\rsf \in L^{\infty}(\Omega)$, and $\Y \approx |\log (\# \T_{\Y})|$, then
\begin{equation}
 \label{eq:curved_1}
 \| \nabla (\ue(\rsf) - V)\|_{L^2(y^{\alpha},\C)} \lesssim | \log (\# \T_{\Y})|^s (\# \T_{\Y})^{-1/(n+1)},
\end{equation}
where the hidden constant is independent of $\ue(\rsf)$, $V$, $\rsf$ and $\T_{\Y}$.}
\end{lemma}
\begin{proof}
\EO{We start with an application of the triangle inequality and the exponential estimate of \cite[Theorem 3.5]{NOS} to deduce that
\begin{multline}
\label{eq:step_1_c} 
\| \nabla (\ue(\rsf) - V)\|_{L^2(y^{\alpha},\C)}  \leq   \| \nabla (\ue(\rsf) - v)\|_{L^2(y^{\alpha},\C)} 
\\ +   \| \nabla (v- V)\|_{L^2(y^{\alpha},\C_{\Y})}
  \lesssim e^{-\sqrt{\lambda_1} \Y / 4} \| \rsf \|_{\Hsd} + \| \nabla (v- V)\|_{L^2(y^{\alpha},\C_{\Y})},
\end{multline}
where $v$ corresponds to the the solution to \eqref{eq:truncated_state} and $\lambda_1$ denotes the first eigenvalue of the operator $\mathcal{L}$. It thus suffices to control the second term on the right hand side of the previous expression. To accomplish this task, we write
\begin{equation*}
 \|\nabla(v -V ) \|_{L^2(y^{\alpha},\C_{\Y})} =  \|\nabla(v -V ) \|_{L^2(y^{\alpha},\C_{\T} )}
 + \|\nabla(v -V ) \|_{L^2(y^{\alpha},\C_{\Y} \backslash \C_{\T} )}
\end{equation*}
and estimate each term separately. We begin with $\|\nabla(v -V(\rsf) ) \|_{L^2(y^{\alpha},\C_{\T} )}$, where, we recall that $\C_{\T} = \Omega_{\T} \times (0,\Y)$. We denote by $W_{e}$ the extension by zero of $W \in \V(\T_{\Y})$ to $\C$. Then, given $W \in \V(\T_{\Y})$, we invoke problem \eqref{fd_a_new} and problem \eqref{eq:truncated_state} with $\phi = W_{e}$ to arrive at the following Galerkin orthogonality property:
\[
 a_{\T}( v - V, W )  = 0 \quad \forall W \in \V(\T_{\Y}).
\]
This immediately yields
\[
 \| \nabla ( v - V ) \|_{L^2(y^{\alpha},\C_{\T})} = \inf_{W \in \V(\T_{\Y})} \| \nabla(v -W)\|_{L^2(y^{\alpha},\C_{\T})},
\]
which, in view of the piecewise polynomial interpolation results of \cite{NOS,NOS3}, and the regularity results of \cite[Theorem 2.7]{NOS}, implies the following quasi--optimal error estimate in terms of degrees of freedom:
\begin{equation}
\label{eq:step_2_c}
 \| \nabla (v - V ) \|_{L^2(y^{\alpha},\C_{\T})} \lesssim |\log N|^s N^{-1/(n+1)} \| \rsf \|_{\Ws};
\end{equation}
we refer the reader to the proof of \cite[Theorem 5.4]{NOS} for details and remark that, the results of \cite[Theorem 2.7]{NOS} are valid under the assumption $\rsf \in \Ws$.}
 
\EO{We now bound $\|\nabla(v -V(\rsf) ) \|_{L^2(y^{\alpha},\C_{\Y} \backslash \C_{\T} )}$: Since $V(\rsf) = 0$ on $ \bar{\Omega} \setminus \Omega_{\T} \times (0,\Y]$, an application of the exponential estimate \cite[Theorem 3.5]{NOS} implies that
\begin{multline}
\label{eq:step_3_c}
\|\nabla( v -V  ) \|_{L^2(y^{\alpha},\C_{\Y} \setminus \C_{\T} )}  = \|\nabla v  \|_{L^2(y^{\alpha},\C_{\Y} \setminus \C_{\T} )} 
\\
\lesssim e^{-\sqrt{\lambda_1}\Y/4 } \| \rsf \|_{\Hsd} + \| \nabla \ue(\orsf)\|_{L^2(y^{\alpha},\C_{\Y} \setminus \C_{\T})}.
\end{multline}
To control the remainder term, we use pointwise estimates for the harmonic extension $\ue(\rsf)$ that relies on the fact that $\rsf \in L^{\infty}(\Omega)$. These estimates are described, for instance, in \cite{MR3348118}: $| \nabla_{x'} \ue(\rsf) | \lesssim y^{-\alpha}$ for $(x',y) \in \C_{\Y}$ (\cite[inequality (6.1)]{MR3348118}) and  $| \partial_{y} \ue(\rsf) | \lesssim y^{-\alpha}$ for $(x',y) \in \C_{\Y}$ (\cite[inequality (6.2)]{MR3348118}). These estimates imply that}
\begin{multline}
\label{eq:inspection}
 \| \nabla \ue(\orsf)\|^2_{L^2(y^{\alpha},\C_{\Y} \setminus \C_{\T})} = \int_{0}^{\Y} \int_{\Omega\setminus \Omega_{\T}} y^{\alpha}| \nabla_{x'} \ue(\rsf) |^2 \diff x'\diff y \\
  + \int_{0}^{\Y} \int_{\Omega\setminus \Omega_{\T}} y^{\alpha} | \partial_{y} \ue(\rsf) |^2 \diff x'\diff y \lesssim \Y^{1-\alpha} |\Omega \setminus \Omega_{\T}|.
\end{multline}
\EO{Since $|\Omega \setminus \Omega_{\T}| \lesssim h_{\T_{\Omega}}^2$ and $ h_{\T_{\Omega}} \approx N^{-1/(n+1)}$, a collection of all the derived estimates allow us to conclude that
\begin{align*}
  \| \nabla (\ue(\rsf) - V)\|_{L^2(y^{\alpha},\C)} 
  \lesssim e^{-\sqrt{\lambda_1} \Y / 4} \| \rsf \|_{\Hsd} + |\log(\# \T_{\Y})|^s (\# \T_{\Y})^{-1/(n+1)}.
\end{align*}
This, in light of the fact that $\Y \approx |\log(\# \T_{\Y})|$, implies the desired estimate \eqref{eq:curved_1} and concludes the proof.}
\end{proof}

\begin{remark}[regularity of $\rsf$]
Examining the proof of Lemma \ref{le:curved_domains_1}, we realize that the critical step where the $L^{\infty}(\Omega)$--regularity of $\rsf$ is needed is \eqref{eq:inspection}. This assumption guarantees the pointwise estimates for $\ue(\rsf)$ used to control its energy on $\C_{\Y} \backslash \C_{\T}$.
\end{remark}

We present an improvement on Lemma \ref{le:curved_domains_1}: an error estimate that is quasi--optimal in terms of approximation and only requires the $\Ws$--regularity of $\rsf$. This improvement will allow us to derive an $L^2(\Omega)$--error estimate via a duality argument.

\begin{figure}[h!]
\centering
\includegraphics[width=0.35\textwidth]{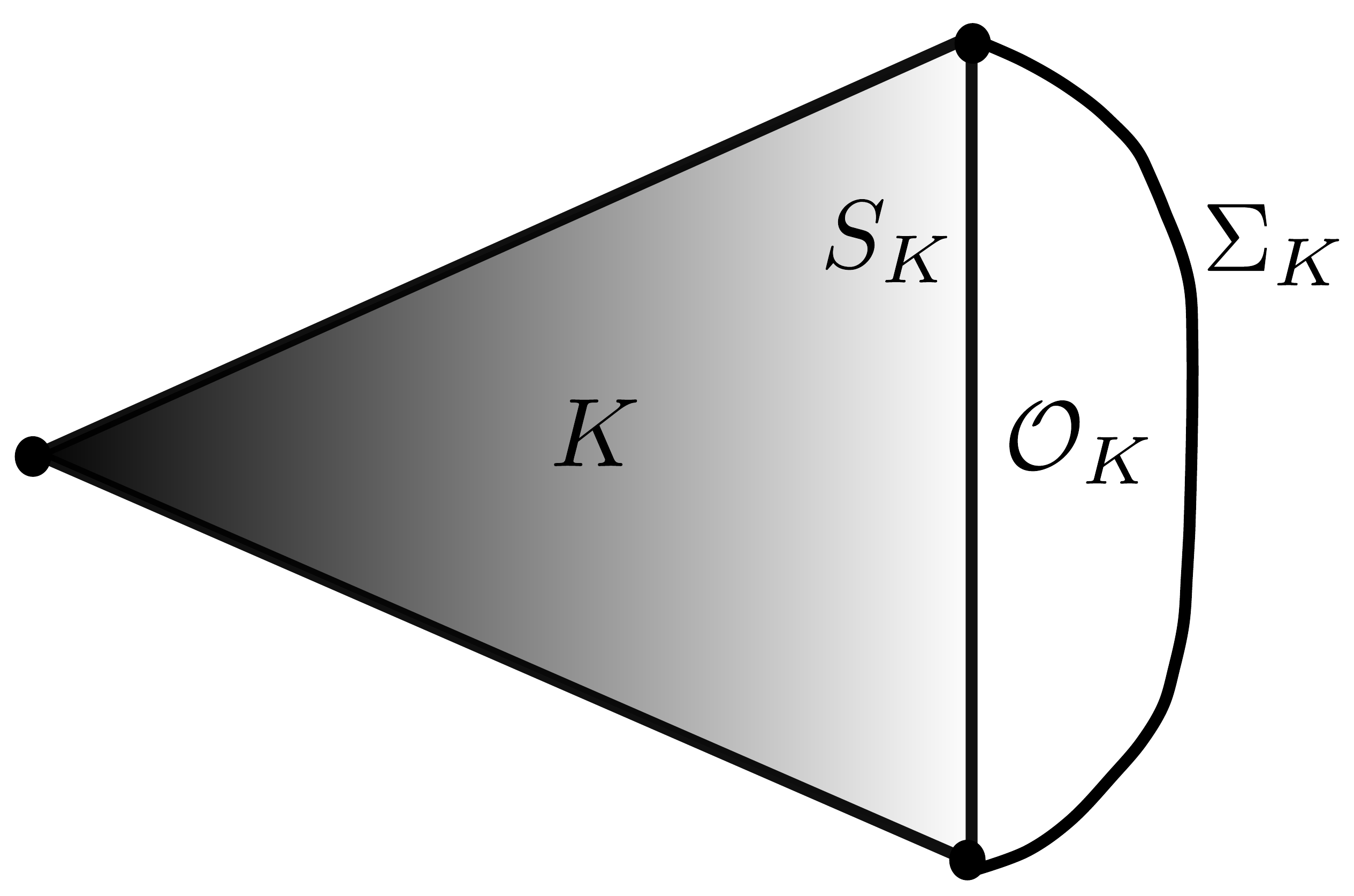} 
\caption{An element $K$, for $n=2$, such that $K \cap \partial \Omega_{\T} \neq \emptyset$ and the corresponding  curved region $\mathcal{O}_K$: $S_K$ denotes the side of $K$ that lies on $\partial \Omega_{\T}$, $\Sigma_K$ the corresponding arc formed by the curved boundary $\partial \Omega$ and $\mathcal{O}_K$ the region bounded by $\Sigma_K$ and $S_K$.}
\label{fig:curved}
\end{figure}

\begin{theorem}[energy--error estimate on curved domains]
\label{le:curved_domains_2}
\EO{If $\rsf \in \Ws$, then, under the framework of Lemma \ref{le:curved_domains_1}, we have that
\begin{equation}
 \label{eq:curved_2}
 \| \nabla (\ue(\rsf) - V )\|_{L^2(y^{\alpha},\C)} \lesssim | \log (\# \T_{\Y})|^s (\# \T_{\Y})^{-1/(n+1)} \| \rsf \|_{\Ws},
\end{equation}
where the hidden constant is independent of $\ue(\rsf)$, $V$, $\rsf$, and $\T_{\Y}$.}
\end{theorem}
\begin{proof}
\EO{In view of the estimates \eqref{eq:step_1_c}--\eqref{eq:step_3_c}, we conclude that it suffices to bound the $L^2$--weighted norm of $\nabla \ue(\rsf)$ on $\C_{\Y} \backslash  \C_{\T}$. To accomplish this task, we proceed by a density argument that is inspired in the techniques developed in the proof of Lemma 5.2.3 in \cite{MR773854}. We present a proof for $n=2$. Let $K \in \Omega_{\T}$ be a cell such that $K \cap \partial \Omega_{\T} \neq \emptyset$. We denote by $S_{K}$ the side of $K$ that lies on the boundary $\partial \Omega_{\T}$ and define $\mathcal{O}_K$ to be the domain bounded by the side $S_K$ and the arc $\Sigma_K$; see Figure \ref{fig:curved}. We choose local coordinates $(\zeta, \mu)$ such that $\zeta$ is defined along the side $S_{K}$ and $\mu$ is orthogonal to $S_{K}$. We denote by $\eta = \varphi(\zeta)$ the equation of the arc $\Sigma_K$. With this notation at hand, and for a smooth function $\phi$, we thus write the following relation
\[
 \phi(\zeta, \mu , y ) =  \phi(\zeta, \varphi(\zeta), y) + \int_{\varphi(\zeta)}^{\mu} \phi_{\eta} (\zeta, t, y) \diff t,
\]
where $y \in I$, with $I$ denoting an interval of the partition $I_{\Y}$ defined in Section \ref{sec:state_equation}, and $(\zeta,\mu ) \in \mathcal{O}_K$. Consequently, integrating on $I$ and applying the Cauchy--Schwarz inequality, we arrive at
\begin{align*}
\label{eq:integration_y}
 \int_{I} y^{\alpha} \phi^2( \zeta, \mu, y) \diff y & \lesssim \int_{I} y^{\alpha} \phi^2(\zeta, \varphi(\zeta), y) \diff y + h_{\T_{\Omega}}^2 \int_{I}  \int_{\varphi(\zeta)}^{\mu}  y^{\alpha}\phi^2_{\eta} (\zeta, t , y) \diff t \diff y.
\end{align*}
To obtain the previous estimate we have used the $C^2$--regularity of the domain $\Omega$ to conclude that $\textrm{dist}(x,\partial \Omega) \lesssim h^2_{\T_{\Omega}}$ for all $x \in \partial \Omega_{\T}$ and then that $|\varphi(\zeta) - \mu | \lesssim h_{\T_{\Omega}}^2$. Integrating, now, on $\mathcal{O}_K$, we obtain that
\begin{equation}
\label{eq:step_4_c}
\| \phi \|^2_{L^2(y^{\alpha},\mathcal{O}_K \times I)} \lesssim h_{\T_{\Omega}}^2  \| \phi \|^2_{L^2(y^{\alpha},( \partial \mathcal{O}_K \cap \partial \Omega ) \times I)} 
+ h_{\T_{\Omega}}^4 \| \nabla_{x'} \phi \|^2_{L^2(y^{\alpha},\mathcal{O}_K \times I)},
\end{equation}
where we have used, again, that $|\varphi(\zeta) - \mu | \lesssim h_{\T_{\Omega}}^2$. Then, on the basis of a density argument, we obtain, for $i=1,\cdots,n+1$, that}
\begin{multline}
\label{eq:step_5_c}
\| \partial_{x_i} \ue(\rsf) \|^2_{L^2(y^{\alpha},\mathcal{O}_K \times I)} \lesssim h_{\T_{\Omega}}^2  \| \partial_{x_i} \ue(\rsf) \|^2_{L^2(y^{\alpha},(\partial \mathcal{O}_K \cap \partial \Omega ) \times I)} 
\\
+ h_{\T_{\Omega}}^4 \| \nabla_{x'} \partial_{x_i} \ue(\rsf) \|^2_{L^2(y^{\alpha},\mathcal{O}_K \times I)}.
\end{multline}

\EO{We control $\|\partial_{x_i} \ue(\rsf) \|_{L^2(y^{\alpha},\partial \Omega  \times I)}$ for $i=1,\dots,n$. To accomplish this task, we use that $H^1(y^{\alpha},\C_{\Y})$ can be equivalently defined as the set of measurable functions $w$ such that $w \in H^1(\Omega \times (s,t))$ for all $0< s< t < \Y$ and for which the seminorm $\| \nabla w \|_{L^2(y^{\alpha},\C_{\Y})}$ is finite. Thus, if $w \in H^1(y^{\alpha},\C_{\Y})$, an application of a standard trace inequality over a dyadic partition that covers the interval $I$ implies
\[
 \int_{\partial \Omega \times I } y^{\alpha} w^2 \lesssim \int_{\Omega \times I} y^{\alpha} | \nabla w|^2.
\]}

\EO{Summing up \eqref{eq:step_5_c} over all the cells $K$ such that $K \cap \partial \Omega_{\T} \neq \emptyset$ and over all the cells $I \in \mathcal{I}_{\Y}$, using the previous estimate and applying the regularity estimate \eqref{reginx} (see  \cite[Theorem 2.7]{NOS}) we obtain, for $i=1,\dots, n$, the estimate
\begin{multline}
\label{eq:step_6_c}
\| \partial_{x_i} \ue(\rsf) \|^2_{L^2(y^{\alpha}, \C_{\Y}\backslash \C_{\T} )} \lesssim h_{\T_{\Omega}}^2  \| \partial_{x_i} \nabla \ue(\rsf) \|^2_{L^2(y^{\alpha}, \C_{\Y} )} 
\\
+ h_{\T_{\Omega}}^4 \| \nabla_{x'} \partial_{x_i} \ue(\rsf) \|^2_{L^2(y^{\alpha},\C_{\Y}\backslash \C_{\T})}
\lesssim h_{\T_{\Omega}}^2 \| \rsf \|^2_{\Ws} \lesssim (\# \T_{\Y})^{-\frac{2}{n+1}}\| \rsf \|_{\Ws}.
\end{multline}
In the previous inequality we have used that $h_{\T_{\Omega}} \approx (\# \T_{\Omega})^{-1/n}$ and that $\# \T_{\Y} \approx \# \T_{\Omega} \cdot M \approx M^{n+1}$, where $M$ denotes the degrees of freedom of the mesh $\mathcal{I}_{\Y}$.}

\EO{The estimate of the term $\|\partial_{x_i} \ue(\rsf) \|_{L^2(y^{\alpha},(\partial \mathcal{O}_K \cap \partial \Omega ) \times I)}$ for $i=n+1$ in \eqref{eq:step_5_c}, follows from the formula \eqref{eq:representation_formula} and the asymptotic estimates for the modified Bessel function of the second kind $K_s$. In fact, \cite[equations (2.28)--(2.29)]{NOS} imply that 
\[
  \frac{\diff }{\diff y }\psi_k(y) = - c_s \sqrt{\lambda_{k}}\left(\sqrt{\lambda_k} y \right)^s K_{s-1}\left(\sqrt{\lambda_k} y \right), \quad \frac{\diff }{\diff y }\psi_k(y) \approx y^{-\alpha} \lambda_k^{s/2} \textrm{ as } y \downarrow 0^{+}.
\]
For brevity we leave the details to the reader. In view of this obtained estimate and \eqref{eq:step_6_c}, we conclude, for $i =1,\dots, n+1 $, that
\[
\| \partial_{x_i} \ue(\rsf) \|^2_{L^2(y^{\alpha}, \C\backslash \C_{\T} )} 
\lesssim h_{\T_{\Omega}}^2 \| \rsf \|^2_{\Ws} \lesssim (\# \T_{\Y})^{-\frac{2}{n+1}}\| \rsf \|_{\Ws},
\]
which, in view of the estimates \eqref{eq:step_1_c}--\eqref{eq:step_3_c}, yields the desired result \eqref{eq:curved_2}.}
\end{proof}

We now provide an $L^2(\Omega)$--error estimate for $\tr (\ue(\rsf) - V)$.
\begin{theorem}[$L^2(\Omega)$--error estimate on curved domains]
\label{le:curved_domains_3}
\EO{Let $\ue(\rsf) \in \HL(y^{\alpha},\C)$ be the solution to \eqref{eq:extension_weak} with $\zsf$ replaced by $\rsf$ and let $V$ be the solution to \eqref{fd_a_new}. If $\Omega$ is a convex $C^2$ domain, $A \in C^{0,1}(\bar \Omega)$, $\rsf \in \Ws$, and $\Y \approx |\log (\# \T_{\Y})|$, then
\begin{equation}
 \label{eq:curved_3}
 \| \tr(\ue(\rsf) - V)\|_{L^2(\Omega)} \lesssim | \log (\# \T_{\Y})|^{2s} (\# \T_{\Y})^{-(1+s)/(n+1)} \| \rsf \|_{\Ws},
\end{equation}
where the hidden constant is independent of $\ue(\rsf)$, $V$, $\rsf$ and $\T_{\Y}$.}
\end{theorem}
\begin{proof}
 \EO{We start with an application of the triangle inequality and the trace estimate \eqref{Trace_estimate} combined with the exponential estimate \cite[Theorem 3.5]{NOS} to derive that}
 \begin{align*}
  \| \tr (\ue(\rsf) - V)\|_{L^2(\Omega)} & \leq   \| \tr (\ue(\rsf) - v)\|_{L^2(\Omega)} +   \| \tr (v- V)\|_{L^2(\Omega)}
  \\
  & 
  \lesssim e^{-\sqrt{\lambda_1} \Y / 4} \| \rsf \|_{\Hsd} + \| \tr (v- V)\|_{L^2(\Omega)}.
\end{align*}

\EO{To control the remainder term, we define $e = \tr(  v- V )$ and $\mathcal{E} = v- V $, and denote by  $\mathfrak{P}_{\T_{\Omega}} : L^2(\Omega) \rightarrow \U(\T_{\Omega})$ the standard $L^2$--projection. With this notation at hand, we write $ \| e \|_{L^2(\Omega)} \leq \| e - \mathfrak{P}_{\T_{\Omega}}e \|_{L^2(\Omega)} + \|\mathfrak{P}_{\T_{\Omega}}e \|_{L^2(\Omega_{\T})}$. The control of the first term follows from interpolation on Sobolev spaces \cite[Lemma 2.4]{MR1972729}, standard interpolation results, the trace estimate \eqref{Trace_estimate} and \eqref{eq:curved_1}. In fact,
\[
  \|e - \mathfrak{P}_{\T_{\Omega}} e \|_{L^2(\Omega)} \lesssim h^s_{\T_{\Omega}} \| e \|_{\Hs} \lesssim h^s_{\T_{\Omega}} \| \nabla \mathcal{E} \|_{L^2(y^{\alpha},\C_{\Y})} \lesssim |\log N|^s N^{-\frac{1+s}{n+1}} \| \rsf\|_{\Ws}.
\]
It thus suffices to control the term $\|\mathfrak{P}_{\T_{\Omega}}e \|_{L^2(\Omega_{\T})}$. We argue by duality and define
\[
\theta \in \HL(y^{\alpha},\C_{\Y}): \quad a_{\Y}(\phi,\theta) = ( \mathfrak{P}_{\T_{\Omega}}e, \tr \phi)_{L^2(\Omega)} \quad \forall \phi \in \HL(y^{\alpha},\C_{\Y}).
\]
Then, by setting $\phi = \mathcal{E}$ and using that $\tr \mathcal{E} = e$ together with the definition of $\mathfrak{P}_{\T_{\Omega}}$, we obtain that
\begin{equation}
\label{eq:L2_aux_1}
\|\mathfrak{P}_{\T_{\Omega}}e \|_{L^2(\Omega)}^2 = a_{\Y}(\mathcal{E},\theta) \lesssim \| \nabla \mathcal{E} \|_{L^2(y^{\alpha},\C_{\Y})}\| \nabla(\theta - \Theta) \|_{L^2(y^{\alpha},\C_{\Y})},
\end{equation}
where $\Theta$ denotes the finite element approximation of $\theta$ on the space $\V(\T_{\Y})$. Applying Theorem \ref{le:curved_domains_2}, together with the regularity estimates \eqref{reginx}--\eqref{reginy} we obtain
\[
 \| \nabla(\theta - \Theta) \|_{L^2(y^{\alpha},\C_{\Y})} \lesssim | \log (\# \T_{\Y})|^{s} (\# \T_{\Y})^{-1/(n+1)} \|\mathfrak{P}_{\T_{\Omega}}e   \|_{\Ws}.
\]
Since the family $ \{ \T_{\Omega} \}$ is quasi--uniform, an inverse inequality implies the estimate $\|\mathfrak{P}_{\T_{\Omega}}e   \|_{\Ws} \lesssim h^{s-1}_{\T_{\Omega}} \|\mathfrak{P}_{\T_{\Omega}}e   \|_{L^2(\Omega)}$. This, combined with \eqref{eq:L2_aux_1} and \eqref{eq:curved_2}, yields
\[
\|\mathfrak{P}_{\T_{\Omega}}e \|_{L^2(\Omega)}^2 \lesssim | \log (\# \T_{\Y})|^{2s} (\# \T_{\Y})^{-(1+s)/(n+1)} \|\mathfrak{P}_{\T_{\Omega}}e   \|_{L^2(\Omega)} \| \rsf \|_{\Ws},
\]
which implies \eqref{eq:curved_3} and concludes the proof.}
\end{proof}

We conclude with the following error estimates for the approximation of problem \eqref{eq:fractional} that follow from an application of Theorems \ref{le:curved_domains_2} and \ref{le:curved_domains_3}. To state these results, we define the following approximation of the solution to \eqref{eq:fractional} with $\zsf$ replaced by $\rsf$:
\begin{equation}
 \label{eq:U}
  U = \tr V, 
\end{equation}
where $V$ denotes the solution to the discrete problem \eqref{fd_a_new}. The discrete space for $ U$ is defined as $ \mathbb{U}_{\T_{\Omega}}:= \tr \V(\T_{\Y})$.

\begin{theorem}[error estimates for fractional diffusion]
\label{le:curved_domains_s}
\EO{Let $\usf(\rsf) \in \Hs$ be the solution to problem \eqref{eq:extension_weak} with $\zsf$ replaced by $\rsf$, and let $U \in \mathbb{U}(\T_{\Omega})$ be its numerical approximation defined by \eqref{eq:U}. If $\Omega$ is a convex $C^2$ domain, $A \in C^{0,1}(\bar \Omega)$ and $\Y \approx |\log (\# \T_{\Y})|$, then we have the following error estimates:
\begin{equation}
 \label{eq:curved_s_1}
 \| \usf(\rsf) - U \|_{\Hs} \lesssim | \log (\# \T_{\Y})|^s (\# \T_{\Y})^{-1/(n+1)}\| \rsf \|_{\Ws},
\end{equation}
and 
\begin{equation}
 \label{eq:curved_s_2}
 \| \usf(\rsf) - U \|_{L^2(\Omega)} \lesssim | \log (\# \T_{\Y})|^{2s} (\# \T_{\Y})^{-(1+s)/(n+1)}\| \rsf \|_{\Ws},
\end{equation}
where the hidden constants are independent of $\ue(\rsf)$, $V$, $\rsf$ and $\T_{\Y}$.}
\end{theorem}

\section{A fully discrete scheme for the optimal control problem}
\label{sec:fd}
The results of previous sections are important in two aspects. First, we were able to replace the original fractional optimal control problem by an equivalent one that involves a local operator and is posed on the semi-infinite cylinder $\C$. Then, we considered the truncated optimal control problem, that is posed on the
bounded domain $\C_{\Y}$, while just incurring in an exponentially small error in the process; see Lemma \ref{LE:exp_convergence}. Since in Section \ref{sec:state_equation} we have proposed and analyzed a finite element discretization to approximate the solution to the truncated state equation \eqref{eq:truncated_state} on curved domains, it remains to propose an efficient solution technique to solve the truncated optimal control problem. This is the content of this section.

We assume that $\Omega$ is a convex $C^2$ domain, and introduce a new fully--discrete scheme that is based on the approximation of the optimal control by piecewise linear functions on quasi--uniform meshes. This is in contrast to \cite{AO}, where the optimal control is discretized with piecewise constants functions. To approximate the optimal state, we use the first--degree tensor product finite elements on anisotropic meshes described in Section \ref{subsec:apriori_fractional}. To be precise, the scheme reads as follows:
$
\text{min } J(\tr V , Z ),  
$
subject to the discrete state equation
\begin{equation}
\label{fd_a}
a_\T (V,W) =  ( Z, \tr W )_{L^2(\Omega)} \quad \forall W \in \V(\T_{\Y}),
\end{equation}
and the discrete control constraints
$
Z \in \mathbb{Z}_{ad}(\T_{\Omega}),
$
where
\begin{equation*}
\mathbb{Z}_{ad}(\T_{\Omega} ) = \Zad \cap \left\{ Z \in L^{\infty}(\Omega_{\T}):
             Z \in C( \bar{\Omega}_{\T} ), Z|_K \in \mathbb{P}_1(K) \, \forall K, \, Z|_{ \bar{\Omega}\setminus \Omega_{\T}}=0 \right\},
\end{equation*}
where $a_{\T}$ is defined in \eqref{fd_a_new}, $K \in \T_{\Omega}$, the discrete space $\V(\T_{\Y})$ is defined in \eqref{eq:V_curved}, and $\mathbb{P}_1(K)$ corresponds to the space of polynomials of total degree at most 1. For convenience, we will refer to the problem previously defined as the \emph{fully--discrete optimal control problem.}

We denote by $(\bar{V}, \bar{Z}) \in \V(\T_\Y) \times  \mathbb{Z}_{ad}(\T_{\Omega})$ the optimal pair solving the fully-discrete optimal control problem. If we define 
\begin{equation}
\label{Hinze_Ufd}
 \bar{U}:= \tr \bar{V},
\end{equation}
we obtain a fully discrete approximation $(\bar{U},\bar{Z}) \in  \U(\T_{\Omega}) \times \mathbb{Z}_{ad}(\T_{\Omega})$ of the optimal pair $(\ousf,\ozsf) \in \Hs \times \Zad$ solving the fractional optimal control problem \eqref{eq:J}--\eqref{eq:cc}. We recall that the finite element space $\U(\T_{\Omega})$ is defined as $ \U(\T_{\Omega})= \tr \V(\T_{\Y})$.

\begin{remark}[locality] 
\rm 
The main advantage of the fully--discrete optimal control problem is that involves the local problem \eqref{fd_a} as state equation. In addition, it can handle multi--dimensions easily and efficiently -- a highly desirable feature.
\end{remark}

We define the discrete operator $\mathbf{H}_{\T_{\Y}}: \mathbb{Z}_{ad}(\T_{\Omega}) \rightarrow \mathbb{U}(\T_{\Omega})$ such that, for $Z \in \mathbb{Z}_{ad}(\T_{\Omega})$, it associates $\tr V \in \mathbb{U}(\T_{\Omega})$, where $V = V(Z) \in \mathbb{V}(\T_{\Y})$ solves \eqref{fd_a}. With this operator at hand, we define the discrete and reduced cost functional 
\begin{equation}
 \label{eq:red_Jd}
 j_{\T_{\Y}}: \Zad  \rightarrow \mathbb{R}, \qquad j_{\T_{\Y}}(\rsf) = J(\rsf, \mathbf{H}_{\T_{\Y}} \rsf),
\end{equation}
which is continuous and convex. In addition, the second derivative $j_{\T_{\Y}}''(\qsf)(\rsf,\rsf)$ does not depend on $\qsf$ and is positive definite. i.e.,
\begin{equation}
 \label{eq:coercivity_discrete}
 j_{\T_{\Y}}''(\qsf)(\rsf,\rsf) \geq \vartheta \| \rsf \|^2_{L^2(\Omega)} \quad \forall \rsf \in L^2(\Omega).
\end{equation}
This property will be important in the error analysis provided in Section \ref{subsec:apriori}. 

We also define the discrete adjoint state $P= P(Z) \in \V(\T_\Y)$ as the solution to
\begin{equation}
\label{fd_adjoint}
a_\T( P,W) = ( \tr V - \usfd , \textrm{tr}_{\Omega} W )_{L^2(\Omega)} \quad \forall W \in \V(\T_{\Y}).
\end{equation}

We present the following result concerning the existence and uniqueness of the optimal control together with the first order necessary and sufficient optimality conditions for the fully discrete optimal control problem. 

\begin{theorem}[existence, uniqueness and optimality system]
The fully discrete optimal control problem has a unique solution $(\bar{V}, \bar{Z})$. The optimality system
\begin{equation}
\begin{dcases}
 \bar{V} = \bar{V}(\bar{Z}) \in \V(\T_\Y) \textrm{ solution of } \eqref{fd_a}, \\
 \bar{P} = \bar{P}(\bar{Z}) \in \V(\T_\Y) \textrm{ solution of } \eqref{fd_adjoint}, \\
 \bar{Z} \in \mathbb{Z}_{ad}(\T_{\Omega}), \quad  (\tr \bar{P} + \vartheta \bar{Z}, Z- \bar{Z})_{L^2(\Omega)} \geq 0 
 \quad \forall Z \in \mathbb{Z}_{ad}(\T),
\end{dcases}
\label{fd_op}
\end{equation}
holds. These conditions are necessary and sufficient.
\end{theorem}

In the next subsection, we will provide an a priori error analysis for the fully discrete optimal control problem. The analysis relies on the regularity properties of the optimal control $\orsf$ derived in Propositions \ref{pro:regularity_control_r_Sobolev} and \ref{pro:regularity_control_r_Holder}. 



\subsection{A priori error analysis for the fully discrete scheme}
\label{subsec:apriori}
In view of the results of Lemma \ref{LE:exp_convergence}, to control the difference $\ozsf - \bar{Z}$ in the $L^2(\Omega)$--norm, it suffices to derive an a priori error bound for the term $\| \orsf - \bar{Z} \|_{L^2(\Omega)}$. To accomplish this task, we introduce the Lagrange interpolation operator $I_{\T_{\Omega}}$ as follows: Given a mesh $\T_{\Omega} \in \Tr_{\Omega}$ and a function $\psi \in C(\bar{\Omega}_{\T})$, we define
\begin{equation*}
 I_{\T_{\Omega}} \psi(x') = \sum_{z \in \N(\T_{\Omega})} \psi(z) \phi_z(x') \quad \forall x' \in \Omega_{\T},
\end{equation*}
where $\{\phi_{z}\}$ is the canonical basis of the finite element space $\mathbb{U}(\T_{\Omega}) = \tr \V(\T_{\Y})$; see   \cite{Guermond-Ern} for details. A simple application of the triangle inequality yields
\begin{equation}
\label{eq:split}
 \| \orsf - \bar{Z} \|_{L^2(\Omega)} \leq  \| \orsf - I_{\T_{\Omega}} \orsf \|_{L^2(\Omega)} + \| I_{\T_{\Omega}} \orsf - \bar{Z}\|_{L^2(\Omega)}.
\end{equation}

To estimate the terms that appear in the right hand side of the previous expression, we consider a suitable partition of the mesh $\T_{\Omega}$ that is based on an assumption about the structure of the active sets \cite{MR2302057,MV:08}. We divide $\T_{\Omega}$ in three subsets, that contain the \emph{active cells, inactive cells} and \emph{cells with kinks}, and are defined, respectively, as follows: 
\begin{align}
\label{eq:part_1}
\T_{\Omega}^1 &= \{ K \in \T_{\Omega}: \orsf|_{K} = \asf \textrm{ or } \orsf|_{K} = \bsf \},
\\
\label{eq:part_2}
\T_{\Omega}^2 &= \{ K \in \T_{\Omega}:  \asf < \orsf|_{K} < \bsf \},
\\
\label{eq:part_3}
\T_{\Omega}^3 &= \T_{\Omega} \setminus (\T_{\Omega}^1 \cup \T_{\Omega}^2).
\end{align}
We notice that $\T_{\Omega}^3$ is the set of all the cells of the mesh $\T_{\Omega}$ that contain the free boundary between the active and the inactive set. We assume the following condition on the mesh $\T_{\Omega}$ and the optimal control $\orsf$: there exists a constant $C>0$, which is independent of $h_{\T_{\Omega}}$, such that
\begin{equation}
 \label{eq:assumption}
 \sum_{K \in \T_{\Omega}^3} |K| \leq C h_{\T_{\Omega}}.
\end{equation}
This assumption is valid, for instance, if the boundary of the sets $\{ x' \in \Omega: \orsf(x') = \asf\}$ and $\{ x' \in \Omega: \orsf(x') = \bsf\}$ contain a finite number of rectifiable curves \cite{MV:08,MR:04}. 

We now bound the first term on the right hand side of \eqref{eq:split}. To simplify the exposition of the material, we define
\begin{equation}
\sigma := 
\begin{dcases}
\tfrac{1}{2}, & s \in (0,\tfrac{1}{4}),\\
2s, & s \in [\tfrac{1}{4},\tfrac{1}{2}),\\
\theta,  & s = \sr, \\
1, & s \in (\sr,1),
\end{dcases}
\label{eq:sigma}
\end{equation}
where $\theta < 1$.

\begin{lemma}[interpolation estimate]
\EO{Let $\Omega$ be a convex $C^2$ domain and $A \in C^{0,1}(\bar \Omega)$. If $\usf_d \in L^{\infty}(\Omega) \cap H_0^1(\Omega)$ and $\asf < 0 < \bsf$, then we have}
\begin{equation}
\label{eq:interpolation_control}
 \| \orsf - I_{\T_{\Omega}} \orsf \|_{L^2(\Omega)} \lesssim h_{\T}^{\tfrac{1}{2} + \sigma},
\end{equation}
where $\sigma$ is defined in \eqref{eq:sigma}, $I_{\T_{\Omega}}$ denotes the Lagrange interpolation operator and the hidden constant is independent of the continuous and discrete optimal variables and the mesh $\T_{\Y}$.
\end{lemma}
\begin{proof}
\EO{We begin by invoking Proposition \ref{pro:regularity_control_r_Holder} to conclude that $\orsf \in C(\bar \Omega)$ and then that $I_{\T_{\Omega}} \orsf$ is well--defined.} We now split the square of the error $\| \orsf - I_{\T_{\Omega}} \orsf \|_{L^2(\Omega)}$ into two contributions:
\[
  \| \orsf - I_{\T_{\Omega}} \orsf \|^2_{L^2(\Omega)} =  \| \orsf - I_{\T_{\Omega}}\orsf \|^2_{L^2(\Omega \setminus \Omega_{\T})} +   \| \orsf - I_{\T_{\Omega}} \orsf \|^2_{L^2(\Omega_{\T})}.
\]

The first term on the right hand side of the previous estimate is bounded by using the fact that $I_{\T_{\Omega}} \orsf$ vanishes on $\Omega \setminus \Omega_{\T}$ and that $\orsf \in \Zad \subset L^{\infty}(\Omega)$:
\begin{equation}
\label{eq:step_1} 
\| \orsf - I_{\T_{\Omega}} \orsf \|^2_{L^2(\Omega \setminus \Omega_{\T})} \lesssim |\Omega \setminus \Omega_{\T}| \| \orsf \|^2_{L^{\infty}(\Omega)} \lesssim h_{\T_{\Omega}}^2  \max\{|\asf|,|\bsf| \}^2.
\end{equation}

\EO{We now control the remainder term $\| \orsf - I_{\T_{\Omega}} \orsf \|_{L^2(\Omega_{\T})}$. A first estimate can be derived in view of the results of Proposition \ref{pro:regularity_control_r_Sobolev}. In fact, a standard interpolation estimate implies that
\begin{equation}
\label{eq:step_2}
\| \orsf - I_{\T_{\Omega}} \orsf \|^2_{L^2(\Omega_{\T})} \lesssim h^2_{\T_{\Omega}} | \orsf |^2_{H^1(\Omega)}. 
\end{equation}
We remark that this step requires that $\asf < 0 < \bsf$, in order to guarantee the $H^1$--regularity of the optimal control $\orsf$; see Proposition \ref{pro:regularity_control_r_Sobolev}. However, such an estimate can be improved, for $s \in (\tfrac{1}{4},1)$, by invoking the regularity results, on Sobolev and H\"older spaces, of Propositions \ref{pro:regularity_control_r_Sobolev} and \ref{pro:regularity_control_r_Holder}, respectively.} To accomplish this task, we follow \cite{MR2302057,MV:08} and invoke the partition \eqref{eq:part_1}--\eqref{eq:part_3} of the mesh $\T_{\Omega}$ to write
\begin{align*}
 \| \orsf - I_{\T_{\Omega}} \orsf \|^2_{L^2(\Omega_{\T})} &=  \sum_{j=1}^3 \sum_{K \in \T_{\Omega}^j} \| \orsf - I_{\T_{\Omega}} \orsf \|^2_{L^2(K)} = \mathrm{I} + \mathrm{II} + \mathrm{III}.
\end{align*}

We proceed to control each term separately. We start with the term $\mathrm{I}$: For each cell $K \in \T_{\Omega}^1$ we have that $\orsf - I_{\T_{\Omega}}\orsf |_{K} = 0$, consequently $\mathrm{I} = 0$.

If $K \in \T_{\Omega}^2$, we have that $\asf < \orsf|_{K} < \bsf$. Therefore $\orsf|_{K} = - (1/\vartheta)  \tr \bar{p}  |_{K}$ and then, over such a cell $K$, $\orsf$ and $\tr \bar{p}$ possess the same regularity, which is dictated by Proposition \ref{pro:regularity_control_r_Sobolev}: $\tr \bar{p} \in \mathbb{H}^{\kappa}(\Omega)$, where $\kappa= \min \{1+2s,2 \}$. We then arrive at
\[
 \textrm{II} \lesssim \sum_{K \in \T_{\Omega}^2} h_{K}^{2\kappa} \| \tr \bar{p} \|^2_{\mathbb{H}^{\kappa}(K)} \leq h_{\T_{\Omega}}^{2\kappa} \| \tr \bar{p} \|^2_{\mathbb{H}^{\kappa}(\Omega)}.
\]

We now proceed to control the term $\mathrm{III}$. In view of the assumption \eqref{eq:assumption}, the regularity results of Proposition \ref{pro:regularity_control_r_Holder} and \eqref{eq:reg_c_r} imply, for $s \in (\frac{1}{4},1)$,  that
\[
 \mathrm{III} \leq \sum_{K \in \T_{\Omega}^3} |K| \|\orsf - I_{\T_{\Omega}} \orsf \|^2_{L^{\infty}(K)} \lesssim h_{\T}^{2 \sigma } \|\orsf \|^2_{C^{0,\sigma}(\bar \Omega)} \sum_{K \in \T^3} |K| \lesssim h_{\T}^{1+2\sigma} \|\orsf \|^2_{C^{0,\sigma}(\bar \Omega)},
\]
where $\sigma$ is defined in \eqref{eq:sigma}. We then collect the derived estimates for the terms $\mathrm{I}$, $\mathrm{II}$ and $\mathrm{III}$ and invoke \eqref{eq:step_2} to obtain that
 \begin{equation*}
\| \orsf - I_{\T_{\Omega}} \orsf \|^2_{L^2(\Omega_{\T})} \lesssim h^{1 + 2\sigma}_{\T_{\Omega}},
\end{equation*}
upon realizing that $\kappa > \tfrac{1}{2} + \sigma$ for $s \in (\frac{1}{4},1)$. This, in view of the estimate \eqref{eq:step_1}, implies the desired estimate.
\end{proof}

We now derive an instrumental result that will be important to estimate the second term on the right hand side of \eqref{eq:split}. To accomplish this task, we introduce the following auxiliary problem:
\begin{equation}
\label{Q}
Q \in \V(\T_{\Y}): \quad a_\T( Q , W) = ( \tr \bar{v} - \usfd , \textrm{tr}_{\Omega} W )_{L^2(\Omega)} \quad \forall W \in \V(\T_{\Y}),
\end{equation}
where $\bar{v}$ denotes the solution to problem \eqref{eq:truncated_state} and $\V(\T_{\Y})$ is defined as in \eqref{eq:V_curved}.

\begin{lemma}[auxiliary estimate I]
\label{le:auxiliary_estimate_I}
\EO{Let $\Omega$ be a convex $C^2$ domain and $A \in C^{0,1}(\bar \Omega)$. If $\usf_d \in \mathbb{H}^{1-s}(\Omega)$ and $\asf < 0 < \bsf$, then we have the estimate}
\begin{equation}
\label{eq:aux_estimate1}
\| \tr  \left(\bar{p} - P(\orsf) \right) \|_{L^2(\Omega)} \lesssim |\log (\# \T_{\Y}) |^{2s} (\# \T_{\Y})^{-(1+s)/(n+1)},
\end{equation}
where $\bar{p} = \bar{p}(\orsf)$ denotes the solution to \eqref{eq:truncated_adjoint} and the hidden constant is independent of the continuous and discrete optimal variables and the mesh $\T_{\Y}$.
\end{lemma}
\begin{proof}
We begin with a simple application of the triangle inequality to deduce that
\begin{equation}
\label{eq:aux_2}
 \| \tr ( \bar{p} - P(\orsf) )\|_{L^2(\Omega)} \leq \| \tr ( \bar{p} -  Q )\|_{L^2(\Omega)} + \| \tr ( Q -  P(\orsf) )\|_{L^2(\Omega)}.
\end{equation}

To bound the first term on the right hand side of the previous expression, we use the trace estimate \eqref{Trace_estimate} and the results of Theorem \ref{le:curved_domains_3}. We thus arrive at 
\[
  \| \tr( \bar{p} -  Q ) \|_{L^2(\Omega)} \lesssim \|\nabla(\bar{p} -  Q ) \|_{L^2(y^{\alpha},\C_{\Y})} \lesssim  |\log (\# \T_{\Y}) |^{2s} (\# \T_{\Y})^{-\frac{1+s}{n+1} } \|\tr \bar{v} \|_{\Ws}.
\]
We remark that, in view of the results of Lemma \ref{le:reg_state}, the term  $\|\tr \bar{v} \|_{\Ws}$ is uniformly bounded.
 
We now proceed to bound the second term on the right hand side of \eqref{eq:aux_2}. To accomplish this task, we invoke the trace estimate \eqref{Trace_estimate} and the stability of the problem that $Q - P(\orsf)$ solve to arrive at
\[
 \| \tr ( Q - P(\orsf))\|_{L^2(\Omega)} \lesssim \|\nabla(Q -  P(\orsf) ) \|_{L^2(y^{\alpha},\C_{\Y})}
 \lesssim \|\tr(\bar{v} - V(\rsf))\|_{L^2(\Omega)}.
\]
The remainder term is bounded by an application of of Theorem \ref{le:curved_domains_3}. This yields
\[
 \| \tr ( Q - P(\orsf))\|_{L^2(\Omega)} \lesssim  |\log (\# \T_{\Y}) |^{2s} (\# \T_{\Y})^{-\frac{1+s}{n+1} } \|\orsf \|_{\Ws};
\]
the $\Ws$--norm of $\orsf$ being uniformly bounded is a consequence of the fact that $\orsf \in H_0^1(\Omega)$; see proposition \ref{pro:regularity_control_r_Sobolev}.

Collecting the derived estimates, we arrive at \eqref{eq:aux_estimate1} and conclude the proof.
\end{proof}

With the Lemma \ref{le:auxiliary_estimate_I} at hand, we proceed to estimate the second term on the right hand side of \eqref{eq:split}.

\begin{lemma}[auxiliary estimate II]
\EO{Let $\Omega$ be a convex $C^2$ domain and $A \in C^{0,1}(\bar \Omega)$. If $\usf_d \in H_0^1(\Omega) \cap L^{\infty}(\Omega)$ and $\asf < 0 < \bsf$, then we have}
\begin{equation}
\label{eq:second_estimate}
 \| I_{\T_{\Omega}} \orsf - \bar{Z} \|_{L^2(\Omega)} \lesssim |\log ( \# \T_{\Y} ) |^{2s} h_{\T}^{\tfrac{1}{2} + \sigma},
\end{equation}
where $\sigma$ is defined in \eqref{eq:sigma} and the hidden constant depends only on the problem data.
\label{le:auxiliary_estimate_II}
\end{lemma}
\begin{proof}
We begin with an application of the coercivity property \eqref{eq:coercivity_discrete} of $j_{\T_{\Y}}$. For an arbitrary $\qsf \in L^2(\Omega)$, we have that
\begin{align}
\nonumber
\vartheta \| I_{\T_{\Omega}} \orsf - \bar{Z} \|^2_{L^2(\Omega)} & \leq j''_{\T_{\Y}}(\qsf)(I_{\T_{\Omega}} \orsf - \bar{Z},I_{\T_{\Omega}} \orsf - \bar{Z})
\\
\label{eq:aux_1}
& = j'_{\T_{\Y}}(I_{\T_{\Omega}} \orsf)(I_{\T_{\Omega}} \orsf - \bar{Z}) - j'_{\T_{\Y}}(\bar{Z})(I_{\T_{\Omega}}\orsf - \bar{Z}).
\end{align}

We now invoke the variational inequality of the optimality system \eqref{op_truncated} with $\rsf = \bar{Z} \in \Zad$ to derive that
\[
0 \leq - j'(\orsf)(\orsf - \bar{Z}) =  - j'(\orsf)( I_{\T_{\Omega}}\orsf  - \bar{Z} ) - j'(\orsf)( \orsf  - I_{\T_{\Omega}} \orsf).
\]

On the other hand, by setting $Z = I_{\T_{\Omega}} \orsf \in \mathbb{Z}_{ad}(\T_{\Omega} )$ in the variational inequality of the discrete optimality system \eqref{fd_op}, we arrive at
\[
 -j'_{\T_{\Y}}(\bar{Z})(I_{\T_{\Omega}} \orsf-\bar{Z}) \leq 0.
\]

In light of the previous two estimates, we proceed to control the right hand side of \eqref{eq:aux_1}. In fact, we have that
\begin{align}
\nonumber
 & \vartheta \| I_{\T_{\Omega}} \orsf - Z \|^2_{L^2(\Omega)}  \leq j'_{\T_{\Y}}(I_{\T_{\Omega}} \orsf)(I_{\T_{\Omega}} \orsf - \bar{Z}) -j'(\orsf)( I_{\T_{\Omega}}\orsf  - \bar{Z} )
 - j'(\orsf)( \orsf  - I_{\T_{\Omega}} \bar{\rsf})
 \\
 \nonumber
 & \leq \big[ j'_{\T_{\Y}}(I_{\T_{\Omega}} \orsf)(I_{\T_{\Omega}} \orsf - \bar{Z}) - j'_{\T_{\Y}}(\orsf)(I_{\T_{\Omega}} \orsf - \bar{Z}) \big] 
 + \big[ j'_{\T_{\Y}}(\orsf)(I_{\T_{\Omega}} \orsf - \bar{Z})
 \\
\nonumber
 & -j'(\orsf)( I_{\T_{\Omega}}\orsf  - \bar{Z} ) \big] 
 - j'(\orsf)( \orsf  - I_{\T_{\Omega}} \bar{\rsf}) =: \mathrm{I} + \mathrm{II} - \mathrm{III}.
\end{align}
Then, it suffices to estimate the terms $\mathrm{I}$, $\mathrm{II}$ and $\mathrm{III}$. To bound $\mathrm{I}$, we invoke the Lipschitz continuity property of $j'_{\T_{\Y}}$, which follows from the quadratic structure of $j_{\T_{\Y}}$. This yields, for $\eta \in L^2(\Omega)$, the error estimate
\begin{equation}
\begin{aligned}
\label{eq:step_I}
| \mathrm{I} | & = | j''_{\T_{\Y}} (\eta)( I_{\T_{\Omega}} \orsf - \orsf,I_{\T_{\Omega}} \orsf - \bar{Z} )| 
\\
& = \big| (\mathbf{H}_{\T_{\Y}} (I_{\T_{\Omega}} \orsf - \orsf), \mathbf{H}_{\T_{\Y}}(I_{\T_{\Omega}} \orsf - \bar{Z}) )_{L^2(\Omega)}  + \vartheta ( I_{\T_{\Omega}} \orsf - \orsf,  I_{\T_{\Omega}} \orsf - \bar{Z})_{L^2(\Omega)}\big|
\\
& \lesssim \| I_{\T_{\Omega}}  \orsf - \orsf \|_{L^2(\Omega)} \|  I_{\T_{\Omega}} \orsf - \bar{Z} \|_{L^2(\Omega)} \lesssim h_{\T_{\Omega}}^{\sigma + \frac{1}{2}}\|  I_{\T_{\Omega}} \orsf - \bar{Z} \|_{L^2(\Omega)},
\end{aligned}
\end{equation}
where we have used the $L^2(\Omega)$--continuity of the discrete control--to--state operator $\mathbf{H}_{\T_{\Y}}$ and the error estimate \eqref{eq:interpolation_control}.

We now control the term $\mathrm{II}$. Invoking the definitions \eqref{eq:red_J} and \eqref{eq:red_Jd} of $j$ and $j_{\T_{\Y}}$, respectively, we write
\[
 | \mathrm{II} | = | (\tr (P(\orsf) - \bar{p}), I_{\T_{\Omega}} \orsf - \bar{Z} ) _{L^2(\Omega)} | \leq \| \tr ( P(\orsf) -  \bar{p} )\|_{L^2(\Omega)} \| I_{\T_{\Omega}} \orsf - \bar{Z}\|_{L^2(\Omega)},
\]
and then, by applying the result of Lemma \ref{le:auxiliary_estimate_I}, we obtain that
\begin{equation}
\label{eq:step_II}
 | \mathrm{II} | \lesssim |\log ( \# \T_{\Y} ) |^{2s} h_{\T_{\Omega}}^{1+s} \| I_{\T_{\Omega}} \orsf - \bar{Z}\|_{L^2(\Omega)}.
\end{equation}

We finally bound the term $\mathrm{III}$. To accomplish this task, we write
\begin{align*}
\mathrm{III} = j'(\orsf)( \orsf  - I_{\T_{\Omega}} \orsf) & = (\tr \bar{p} + \vartheta \orsf,  \orsf  - I_{\T_{\Omega}} \orsf)_{L^2(\Omega)} 
\\
& = (\tr \bar{p} + \vartheta \orsf,  \orsf  - I_{\T_{\Omega}} \orsf)_{L^2(\Omega \setminus \Omega_{\T})}  +  \sum_{K \in \T_{\Omega}}\mathrm{III}(K).
\end{align*}
Since $I_{\T_{\Omega}} \orsf$ vanishes on $\Omega \setminus \Omega_{\T}$ and $|\Omega \setminus \Omega_{\T}| \lesssim h_{\T_{\Omega}}^2$, we conclude that
\begin{equation}
\label{eq:52}
| (\tr \bar{p} + \vartheta \orsf,  \orsf)_{L^2(\Omega \setminus \Omega_{\T})} | \lesssim h_{\T_{\Omega}}^{2} (\| \tr \bar{p} \|_{L^{\infty}(\Omega)} + \| \orsf \|_{L^{\infty}(\Omega)}) \| \orsf \|_{L^{\infty}(\Omega)};
\end{equation}
\EO{the uniform control of $\| \tr \bar{p} \|_{L^{\infty}(\Omega)}$ and $\| \orsf \|_{L^{\infty}(\Omega)}$ follows from Proposition \ref{pro:regularity_control_r_Holder}.}

We now estimate each term $\mathrm{III}(K)$ depending on the location of the cell $K \in \T_{\Omega}$ with respect to the partition defined by \eqref{eq:part_1}--\eqref{eq:part_3}.
\begin{enumerate}[$\bullet$]
 \item $K \in \T_{\Omega}^1:$ In this situation, the cell $K$ is \emph{active}. Consequently, $(\orsf - I_{\T_{\Omega}} \orsf) |_{K} = 0$ and then $\mathrm{III}(K) = 0$.
 \item $K \in \T_{\Omega}^2:$ Since the cell $K$ is \emph{inactive}, the optimality condition of the system \eqref{op_truncated} immediately yields $(\tr \bar{p} + \vartheta \orsf)|_{K}=0$. Consequently, $\mathrm{III}(K) = 0$.
 \item $K \in \T_{\Omega}^3$: \EO{A first estimate can be derived by invoking the $L^2$--standard projection operator $\mathfrak{P}_{\T_{\Omega}}: L^2(\Omega_{\T}) \rightarrow \mathbb{U}(\T_{\Omega})$ as follows:}
 \begin{align*}
  \mathrm{III}(K) & = (\tr \bar{p} + \vartheta \orsf - \mathfrak{P}_{\T_{\Omega}}( \tr \bar{p} + \vartheta \orsf ),  \orsf  - I_{\T_{\Omega}} \bar{\rsf})_{L^2(K)} 
  \\
  & +  (\mathfrak{P}_{\T_{\Omega}}( \tr \bar{p} + \vartheta \orsf ),  \mathfrak{P}_{\T_{\Omega}}\orsf  - I_{\T_{\Omega}} \bar{\rsf})_{L^2(K)}.
 \end{align*}
 We thus invoke standard interpolation and inverse estimates, to conclude, for $s \in (0,1)$, that
\[
 \mathrm{III}(K) \lesssim h^2_{\T_{\Omega}} \left( \| \tr \bar{p} + \vartheta \orsf \|_{H^1(K)} \| \orsf \|_{H^1(K)} \right).
\]
\EO{We notice that this estimate would only provide linear orden of convergence for the left hand side of \eqref{eq:second_estimate}. In what follows we present an improvement of this error estimate for $s \in (\frac{1}{4},1)$.} We begin with the estimate
 \[
  | \mathrm{III}(K) | \leq |K| \| \tr \bar{p} + \vartheta \orsf \|_{L^{\infty}(K)}\| \orsf - I_{\T_{\Omega}}\orsf \|_{L^{\infty}(K)}.
 \]
We bound  $\| \tr \bar{p} + \vartheta \orsf \|_{L^{\infty}(K)}$. \EO{To do this, we invoke the fact that $K \in \T_{\Omega}^3$, and then that there exists a point $x_0' \in K$ such that $(\tr \bar{p} + \vartheta \orsf)(x_0') = 0$.} This, in view of Proposition \ref{pro:regularity_control_r_Holder}, the assumption \eqref{eq:reg_c_r} \EO{and the definition of the H\"older--norm,} implies that
\begin{equation}
\begin{aligned}
\label{eq:referee}
 \| \tr \bar{p} + \vartheta \orsf \|_{L^{\infty}(K)} & = \| \tr \bar{p} + \vartheta \orsf - (\tr \bar{p} + \vartheta \orsf)(x_0')\|_{L^{\infty}(K)}
 \\
 & \lesssim h_{\T_{\Omega}}^{\sigma} \| \tr \bar{p} + \vartheta \orsf \|_{C^{0,\sigma}(\bar K)}.
 \end{aligned}
\end{equation}
On the other hand, applying Theorem \ref{th:reg_controlII} and Proposition \ref{pro:regularity_control_r_Holder}, again, we obtain
\[
 \| \orsf - I_{\T_{\Omega}} \orsf \|_{L^{\infty}(K)} \lesssim h_{\T_{\Omega}}^{\sigma} \| \orsf \|_{C^{0,\sigma}(\bar K)}
\]
\end{enumerate}

Consequently, in view of \eqref{eq:52} and the derived estimated for $\mathrm{III}(K)$, we derive
\begin{equation}
\label{eq:step_III}
 |\mathrm{III}| \lesssim h_{\T_{\Omega}}^{2} + \sum_{K \in \T_{\Omega}^3} |\mathrm{III}(K) | \lesssim h_{\T_{\Omega}}^{2} + h_{\T}^{2\sigma} \sum_{K \in \T^3} |K| \lesssim h_{\T_{\Omega}}^{2\sigma+1}.
\end{equation}

Finally, inserting the estimates \eqref{eq:step_I}, \eqref{eq:step_II} and \eqref{eq:step_III} in \eqref{eq:aux_1}, we derive the desired inequality \eqref{eq:second_estimate} upon realizing that $1 + s > \tfrac{1}{2} + \sigma$ for all $s \in (\frac{1}{4},1)$. This concludes the proof. 
\end{proof}

\begin{theorem}[fractional control problem: error  estimate]
\EO{Let $\ozsf$ and $\bar{Z}$ be the optimal controls for the fractional optimal control problem \eqref{eq:J}--\eqref{eq:cc} and the fully discrete optimal control problem, respectively. If $\Omega$ is a convex $C^2$ domain, $A \in C^{0,1}(\bar \Omega)$ and $\usf_d \in L^{\infty}(\Omega) \cap \Ws$, then we have that}
\[
 \| \ozsf - \bar{Z} \|_{L^2(\Omega)} \lesssim  |\log (\# \T_{\Y})|^{2s}(\# \T_{\Y})^{\frac{-1}{n+1} \left( \frac{1}{2}+\sigma \right) }, 
\]
where the hidden constant is independent of the continuous and discrete optimal variables and the mesh $\T_{\Y}$.
\label{th:error_estimates}
\end{theorem}
\begin{proof}
We combine the exponential convergence result of Lemma \ref{LE:exp_convergence} with the estimates \eqref{eq:interpolation_control} and \eqref{eq:second_estimate} to arrive at
\begin{align*}
\| \ozsf - \bar{Z} \|_{L^2(\Omega)} & \leq \| \ozsf - \orsf \|_{L^2(\Omega)} + \| \orsf - \bar{Z} \|_{L^2(\Omega)}
\\
 & \lesssim e^{-\sqrt{\lambda_1} \Y /4} + |\log ( \# \T_{\Y} ) |^{2s} h_{\T}^{\frac{1}{2}+\sigma}.
\end{align*}
A natural choice of $\Y$ comes from equilibrating the two terms on the right hand side of the previous expression: $\Y \approx | \log ( \# \T_\Y ) | $. This gives the desired estimate and concludes the proof.
\end{proof}

\begin{remark}[error estimate for $s \in (0,1)$]
\label{rk:error_estimates}
\EO{For $s \in (\frac{1}{2},1)$, the estimate of Theorem \ref{th:error_estimates} reads
\[
 \| \ozsf - \bar{Z} \|_{L^2(\Omega)} \lesssim |\log ( \# \T_{\Y} ) |^{2s} (\# \T_{\Y})^{-\frac{3}{2(n+1)}}.
\]
Since the family $\{ \T_{\Omega} \}$ is quasi--uniform, $\# \T_{\Omega} \approx M^{n}$ and $\# \T_{\Y} \approx \# \T_{\Omega} \cdot M$, we have that $h_{\T_{\Omega}} \approx (\# \T_{\Omega})^{-1/n} \approx (\# \T_{\Y})^{-1/(n+1)}$. This implies that the previous estimate can be rewritten as 
\[
 \| \ozsf - \bar{Z} \|_{L^2(\Omega)} \lesssim |\log ( \# \T_{\Y} ) |^{2s} h_{\T_{\Omega}}^{\frac{3}{2}},
\]
which, up to the logarithmic term, corresponds to the well--known error estimate for first--degree approximation of the optimal control \cite{MR2302057,MV:08}. In the particular case that $s = 1/2$, Theorem \ref{th:error_estimates} provides a slightly deteriorated error estimate:
\[
 \| \ozsf - \bar{Z} \|_{L^2(\Omega)} \lesssim |\log ( \# \T_{\Y} ) |^{2s} (\# \T_{\Y})^{\frac{-1}{(n+1)} \left( \tfrac{1}{2} + \theta \right)},
\]
for any $\theta < 1$. For $s \in (\frac{1}{4},\frac{1}{2})$, the result of Theorem \ref{th:error_estimates} reads
\begin{equation*}
  \| \ozsf - \bar{Z} \|_{L^2(\Omega)} \lesssim |\log ( \# \T_{\Y} ) |^{2s} (\# \T_{\Y})^{\frac{-1}{(n+1)} \left( \tfrac{1}{2} + 2s \right)},
\end{equation*}
while for $s \in (0,\frac{1}{4})$, Theorem \ref{th:error_estimates} yields a linear order of convergence.}
\end{remark}

\begin{remark}[error estimates: $n \geq 1$]
The advantage of deriving regularity properties of the optimal controls $\ozsf$ and $\orsf$ in H\"older spaces has as a consequence that the provided a priori error analysis holds in a general $n$--dimensional setting, \EO{under the assumption that $\Omega$ is a convex $C^2$ domain and $A \in C^{0,1}(\bar \Omega)$.}
\end{remark}

\section{Numerical experiments}
\label{sec:numerics}

In this section, we conduct a series of numerical experiments that illustrate the performance of the fully discrete scheme proposed in Section \ref{sec:fd} and support our theoretical findings. 

The implementation has been carried out within the MATLAB software library {\it{i}}FEM~\cite{chen2009ifem}. \EO{The stiffness matrices of the discrete systems \eqref{fd_a} and \eqref{fd_adjoint} are assembled exactly, and the respective forcing boundary terms are computed by a quadrature formula which is exact for polynomials of degree 4. The resulting linear system is solved by using the built-in direct solver of MATLAB. To solve the minimization problem, we use the projected Broyden--Fletcher--Goldfrab--Shanno (BFGS) method. The optimization algorithm is terminated when the $\ell^2$--norm of the projected gradient is less than or equal to $10^{-8}$.}


To illustrate the error estimates of Theorem \ref{th:error_estimates} we consider the following exact solution to the fractional optimal control problem \eqref{eq:J}--\eqref{eq:cc}. Let $n=2$, $\vartheta =1$, $\Omega = (0,1)^2$, and $A(x') \equiv 1$ in \eqref{eq:L}. The eigenvalues and eigenfunctions of the operator $\mathcal{L}$ are given by:
\[
\lambda_{k,l} = \pi^2 (k^2 + l^2), \quad \varphi_{k,l}(x_1,x_2) = \sin(k \pi x_1) \sin(l\pi x_2),  
\quad k, l \in \mathbb{N}.
\]
\EO{To construct an exact solution, we consider the following modification or problem \eqref{eq:fractional}. Given $s\in (0,1)$, the forcing term $\fsf$ and the control $\zsf$, the aforementioned modified problem reads: Find $\usf$ such that 
$
\mathcal{L}^s \usf = \fsf + \zsf
$
in $\Omega$ and $\usf = 0$  on  $\partial \Omega$. If $\asf = -0.5$ and $\bsf = 0.5$, 
and 
$
\usf_d = (1 + \vartheta \lambda_{2,2}^s) \sin(2 \pi x_1) \sin(2\pi x_2),
$
then we have that $\ousf = \sin(2 \pi x_1) \sin(2\pi x_2)$, $\opsf = -\vartheta \sin(2 \pi x_1) \sin(2 \pi x_2)$, and 
\[
 \ozsf =\min\left\{ 0.5, \max \left\{ - 0.5 , - \opsf/\vartheta \right\} \right\},
\]
where $\fsf = \lambda_{2,2}^s \sin(2 \pi x_1) \sin(2\pi x_2) - \ozsf$. We notice that $\ozsf \in H_0^1(\Omega) \cap C^{0,1}(\Omega)$ for all the values of $s \in (0,1)$. We remark that under this regularity property of the optimal control, the arguments developed in the proof of the estimates \eqref{eq:interpolation_control} and \eqref{eq:second_estimate} guarantee the error estimate} 
\[
 \| \ozsf - \bar{Z} \|_{L^2(\Omega)} \lesssim |\log ( \# \T_{\Y} ) |^{2s} (\# \T_{\Y})^{-\frac{3}{2(n+1)}}.
\]

\subsection{Piecewise linear versus piecewise constant approximation}

\EO{In this section we explore the advantages of the fully discrete scheme of Section \ref{sec:fd} when solving the fractional optimal control problem; we compare the performance of this scheme ($\mathbb{P}_1$--scheme) with that of the numerical technique investigated in \cite{AO} that is based on piecewise constant approximation of the optimal control ($\mathbb{P}_0$--scheme). Table \ref{table:p1vsp0} show, for different meshes $\T_{\Y}$, the error $E_{\zsf}$ in the control approximation due to the $\mathbb{P}_1$--scheme and the $\mathbb{P}_0$--scheme. Two values of the parameter $s$ are considered: $s = 0.2$ and $s = 0.8$. $\#$DOFs denotes the number of degrees of freedom of $\T_{\Y}$. It can be observed that, for a mesh $\T_{\Y}$ with $\#$ DOFs = $137376$, the error obtained with the $\mathbb{P}_1$--scheme is almost an order in magnitude smaller than the corresponding error due to the $\mathbb{P}_0$--scheme of \cite{AO}.}

\begin{table}
  \begin{center}
    \begin{tabular}{||c||c||c||c||c||c||}
      \hline
      \#  DOFs     & $ E_{\zsf}(\mathbb{P}_0;0.2)$ & $ E_{\zsf}(\mathbb{P}_1;0.2)$ & $ E_{\zsf}(\mathbb{P}_0;0.8)$ & $ E_{\zsf}(\mathbb{P}_1;0.8)$ \\
      \hline
       432     & 0.147712126        & 0.131130828       & 0.1482301425       & 0.1470944750 \\
      \hline
      3146   & 0.083305924       & 0.036668665       &    0.0840901319       & 0.0443202090 \\
      \hline
      10496   &  0.058953277        & 0.020712242      &    0.0588454408      & 0.0241526956 \\
      \hline
      25137  & 0.044253527       & 0.012937511     &    0.0441539905       & 0.0148381456 \\
      \hline
      49348  & 0.035650434       & 0.008967500      &    0.0356800357       & 0.0101409325 \\
      \hline
       85529 & 0.029769320        & 0.007334747      &    0.0297507072       & 0.0080907623 \\
      \hline
       137376 & 0.025419044       & 0.005094037     &   0.0254259814       & 0.0056585074 \\
      \hline
    \end{tabular} 
  \end{center}
\vspace{0.2cm}
\caption{
Experimental errors for both: the fully discrete scheme studied in Section \ref{sec:fd}($\mathbb{P}_1$--scheme) and the $\mathbb{P}_0$--scheme proposed in \cite{AO}. Two different values of the parameter are $s$ are considered: $0.2$ and $0.8$.}
\label{table:p1vsp0}
\end{table}
 
\subsection{Computational convergence rates for $s \in (0,1)$}

In Figure~\ref{fig:grad_rate}, we show the the asymptotic relation $$\|\ozsf - \bar{Z} \|_{L^2(\Omega)} \approx (\# \T_{\Y})^{-\frac{1}{2}}$$ which illustrate the decay rate of our fully-discrete scheme of Section \ref{sec:fd} for $n=2$ and all the choices of the parameter $s$ considered: $s=0.2$, $s=0.4$, $s=0.6$, and $s=0.8$. For $s = 0.6$ and $s=0.8$, the  presented results are in agreement with the error estimate of Theorem \ref{th:error_estimates}. For $s = 0.2$ and $s=0.4$ the results of Figure~\ref{fig:grad_rate} present an experimental convergence rate that is better than the one derived in Theorem \ref{th:error_estimates}. This is due to the fact, in this case, the exact optimal control $\zsf \in H_0^1(\Omega) \cap C^{0,1}(\Omega)$; regularity property that is not provided by Theorem \ref{th:reg_controlII} for $s \in (0,\frac{1}{2})$. Similar results are shown in Figure \ref{fig:grad_rate_2} for the following values of the parameter $s$: $s=0.1$, $s=0.3$, $s=0.5$, $s=0.7$ and $s=0.9$.

\begin{figure}[h!]
\centering
\includegraphics[width=0.55\textwidth]{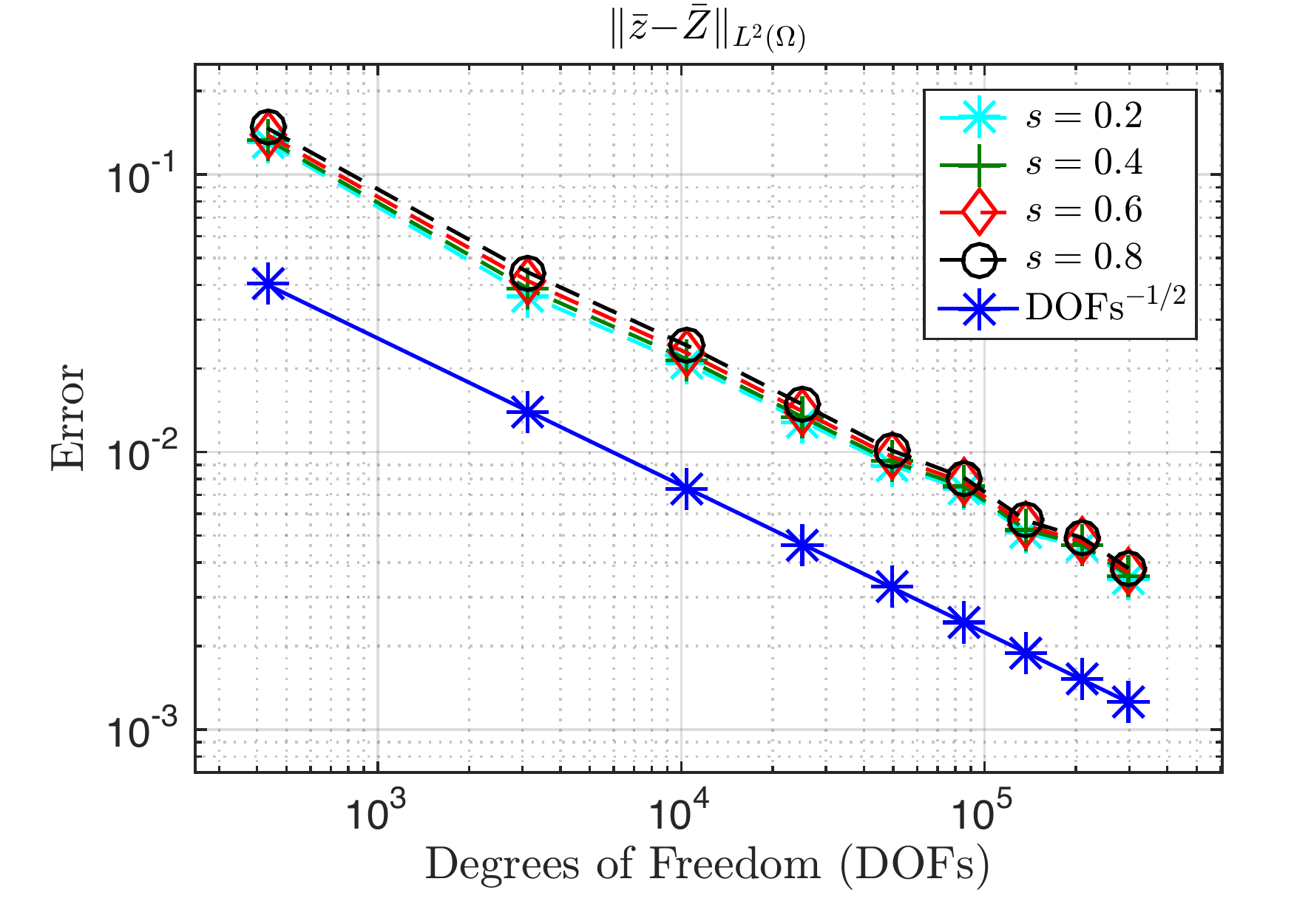} 
\caption{\label{fig:grad_rate}
Computational rates of convergence for the fully discrete scheme proposed in Section \ref{sec:fd} on anisotropic meshes for $n = 2$ and $s=0.2$, $s=0.4$, $s=0.6$ and $s=0.8$. The figure shows the decrease of the $L^2$-norm of the error for the optimal control with respect to $\# \T_{\Y}$. In all the cases we recover the rate $(\# \T_{\Y})^{-1/2}$. }
\end{figure}

\begin{figure}[h!]
\centering
\includegraphics[width=0.55\textwidth]{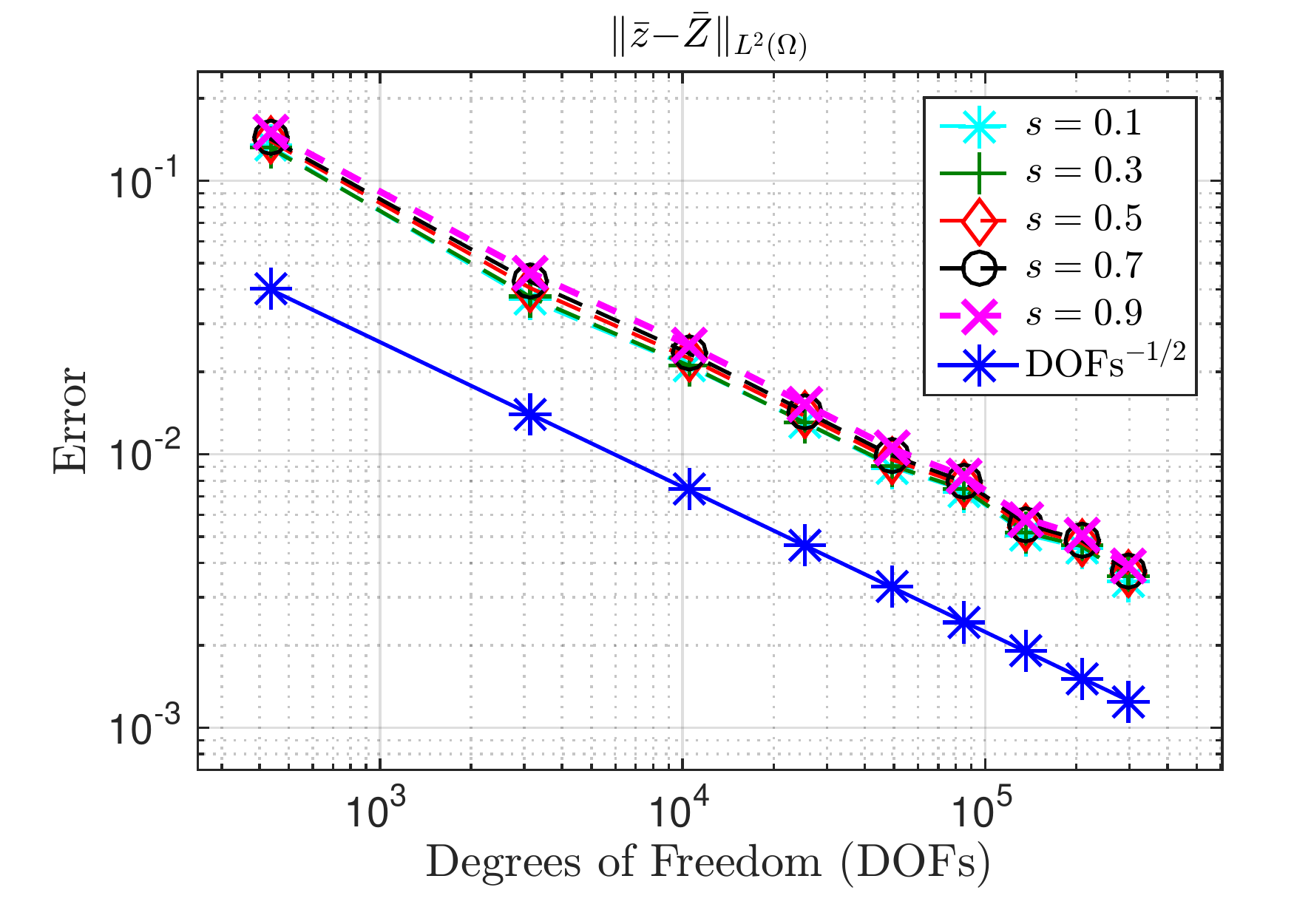} 
\caption{\label{fig:grad_rate_2}
Computational rates of convergence for the fully discrete scheme proposed in Section \ref{sec:fd} on anisotropic meshes for $n = 2$ and $s=0.1$, $s=0.3$, $s=0.5$, $s=0.7$ and $s=0.9$. The figure shows the decrease of the $L^2$-norm of the error for the optimal control with respect to $\# \T_{\Y}$. In all the cases we recover the rate $(\# \T_{\Y})^{-1/2}$.}
\end{figure}

\section*{Acknowledgement}
The author would like to thank A.J. Salgado and P. R. Stinga for fruitful discussions regarding the H\"older regularity of linear problems involving fractional powers of elliptic operators in bounded domains and H. Antil for the help provided with {\it{i}}FEM~\cite{chen2009ifem}. The author would also like to thank B. Vexler for providing some lecture notes that inspired this article.

\bibliographystyle{plain}
\bibliography{biblio}

\end{document}